\documentclass[letterpaper,10pt]{article}
\pdfoutput=1
\usepackage[margin=1in]{geometry}
\usepackage[utf8]{inputenc}
\DeclareUnicodeCharacter{012D}{\u{\i}}

\usepackage{graphicx}
\usepackage{pdflscape}
\usepackage{tikz-cd}
\usepackage[style=alphabetic,%
sorting=anyt,%
maxnames=4,maxalphanames=4,%
backref=true,%
backend=biber,bibencoding=auto]{biblatex}%
\usepackage[noDcommand,leqno]{kpfonts}
\usepackage{amsxtra}
\usepackage{amsthm}
\usepackage{microtype}

\usepackage{comment}
\usepackage{xspace} 
\usepackage[dvipsnames]{xcolor}

\definecolor{theblue}{rgb}{0.2,0.04,0.7}%
\definecolor{thered}{rgb}{0.8,0.04,0.07}%
\definecolor{thegreen}{rgb}{0.06,0.44,0.08}%
\definecolor{thegrey}{gray}{0.5}%
\definecolor{theshade}{gray}{0.92}%
\usepackage[%
colorlinks=true,
citecolor=thegreen,
linkcolor=theblue,
urlcolor=thered,
plainpages=false,
pdfpagelabels,bookmarks]{hyperref}
\usepackage[capitalise,nameinlink]{cleveref} 

\usepackage{enumitem}
\newlist{exaxioms}{enumerate}{2}
\setlist[exaxioms,1]{label=\textup{(E\arabic*)},ref=\textup{(E\arabic*)},
                     leftmargin=*,align=left,widest=(E3)}
\setlist[exaxioms,2]{label=\textup{(\roman*)},
                     ref=\textup{(E\arabic{exaxiomsi})(\roman*)},
                     leftmargin=*,align=left,widest=(iii),
                     topsep=0pt,partopsep=0pt,parsep=0pt}
\crefname{exaxiomsi}{condition}{conditions}
\Crefname{exaxiomsi}{Condition}{Conditions}
\crefname{exaxiomsii}{condition}{conditions}
\Crefname{exaxiomsii}{Condition}{Conditions}
\newlist{conditions}{enumerate}{2}
\setlist[conditions,1]{label=\textup{(\roman*)},ref=\textup{(\roman*)},
                       leftmargin=*,align=left,widest=(iii)}
\setlist[conditions,2]{label=\textup{(\alph*)},
                       ref=\textup{(\roman{conditionsi})(\alph*)},
                       leftmargin=*,align=left}
\crefname{conditionsi}{condition}{conditions}
\Crefname{conditionsi}{Condition}{Conditions}
\crefname{conditionsii}{condition}{conditions}
\Crefname{conditionsii}{Condition}{Conditions}

\usepackage{nicefrac}

\swapnumbers
\newtheorem*{theorem*}{Theorem}
\newtheorem{theorem}{Theorem}[section]
\newtheorem{proposition}[theorem]{Proposition}
\newtheorem*{proposition*}{Proposition}
\newtheorem{lemma}[theorem]{Lemma}
\newtheorem{corollary}[theorem]{Corollary}
\newtheorem{lemma-def}[theorem]{Lemma-Definition}

\theoremstyle{definition}
\newtheorem{definition}[theorem]{Definition}
\newtheorem{example}[theorem]{Example}
\newtheorem*{example*}{Example}
\newtheorem{construction}[theorem]{Construction}
\newtheorem*{construction*}{Construction}
\newtheorem{remark}[theorem]{Remark}
\newtheorem*{remark*}{Remark}

\newtheorem*{variant*}{Variant}
\newtheorem{notation}[theorem]{Notation}
\newtheorem*{notation*}{Notation}

\numberwithin{equation}{section}%

\crefname{construction}{Construction}{Constructions}
\crefname{notation}{Notation}{Notations}
\crefname{lemma-def}{Lemma-Definition}{Lemma-Definitions}
\crefname{claim}{Claim}{Claims}
\crefname{conjecture}{Conjecture}{Conjectures}
\crefname{remarks}{Remarks}{Remarks}
\crefname{variant}{Variant}{Variants}
\newcommand{\cA}{\mathcal{A}}
\newcommand{\cC}{\mathcal{C}}
\newcommand{\cD}{\mathcal{D}}
\newcommand{\cE}{\mathcal{E}}
\newcommand{\cG}{\mathcal{G}}
\newcommand{\cH}{\mathcal{H}}
\newcommand{\cI}{\mathcal{I}}
\newcommand{\cL}{\mathcal{L}}
\newcommand{\cM}{\mathcal{M}}
\newcommand{\cP}{\mathcal{P}}
\newcommand{\cS}{\mathcal{S}}

\newcommand{\cW}{\mathcal{W}}

\newcommand{\NN}{\mathbb{N}}
\newcommand{\ZZ}{\mathbb{Z}}
\newcommand{\Einf}{\mathbb{E}_{\infty}}

\DeclareMathOperator{\Fun}{Fun}
\DeclareMathOperator{\ExFun}{Ex-Fun}
\DeclareMathOperator{\Ex}{Ex}
\DeclareMathOperator{\Exact}{Exact}
\DeclareMathOperator{\Map}{Map}

\DeclareMathOperator{\Fil}{Fil}
\DeclareMathOperator{\Gap}{Gap}
\DeclareMathOperator{\Ch}{Ch}
\DeclareMathOperator{\cofib}{cofib}
\DeclareMathOperator{\colim}{colim}
\DeclareMathOperator{\totcof}{tot-cof}

\DeclareMathOperator{\Ar}{Ar}
\DeclareMathOperator{\Obj}{Obj}
\DeclareMathOperator{\Id}{Id}
\DeclareMathOperator{\id}{id}
\newcommand{\Sp}{\mathrm{Sp}}
\newcommand{\Mon}{\mathrm{Mon}}
\newcommand{\Grp}{\mathrm{Grp}}
\newcommand{\Fin}{\mathrm{Fin}}
\newcommand{\Kan}{\mathrm{Kan}}
\newcommand{\Cat}{\mathrm{Cat}}
\newcommand{\sSet}{\mathrm{sSet}}

\newcommand{\op}{\mathrm{op}}                 
\newcommand{\ho}{\mathrm{h}}                  
\newcommand{\nerve}{\operatorname{N}}         
\newcommand{\sN}{\tilde{\operatorname{N}}}
\newcommand{\Chcoh}{\Ch^{\mathrm{coh}}}       

\newcommand{\boxinf}{\square_{\infty}}
\newcommand{\boxinfop}{\square_{\infty}^{\op}}
\newcommand{\skc}{\mathrm{sk}^{\mathrm{c}}}
\newcommand{\skch}{\mathrm{sk}^{\mathrm{c}}_{\mathrm{ch}}}

\title{On the Mac Lane $Q$-Construction for Exact $\infty$-Categories}
\date{}
\author{Ettore Aldrovandi\thanks{\url{aldrovandi@math.fsu.edu}} \and
  Arash Karimi\thanks{\url{akarimi@fsu.edu}} \and
  \small{Department of Mathematics, Florida State University, Tallahassee, FL 32306-4510}}

\begin{document}
\maketitle

\begin{abstract} 
We extend McCarthy's stabilization construction to exact $\infty$-categories. For any functor from exact $\infty$-categories to a fixed stable $\infty$-category $\cA$, we construct a coherent chain complex of cubical objects in $\cA$. Objectwise total homology reconstructs the associated cubical object, whose alternating-face complex in $\mathrm h\cA$ recovers the corresponding $Q'$-complex in McCarthy's construction. This provides a higher-categorical extension of the framework originating in Mac Lane's cubical $Q$-complex for computing the stable homology of abelian groups.
\end{abstract}

\setcounter{tocdepth}{2}
\tableofcontents

\phantomsection  
\addcontentsline{toc}{section}{Introduction}  
\section*{Introduction}
\label{sec:introduction}

Eilenberg and Mac Lane's cubical $Q$-complex\footnote{Not to be confused with Quillen's $Q$-construction.} computes the stable homology of abelian groups and it has exceptionally good multiplicative properties (see \cite[Chap.\ 13]{MR1600246}).  In \cite{Main-McCarthy}, McCarthy introduced a modified $Q$-construction for exact categories. Building on Mac Lane’s framework, it produces an explicit chain complex computing the spectrum homology of the algebraic $K$-theory spectrum of an exact category. Unlike with Quillen and Waldhausen's constructions where the higher $K$-groups are by definition homotopy groups, here the higher $K$-groups are computed through a complex. In \cite[Section~5]{Main-McCarthy} the coefficients are allowed to be an arbitrary spectrum, and the $Q$-complex becomes the $E^{1}$-page of an Atiyah--Hirzebruch spectral sequence converging to the corresponding generalized homology of the $K$-theory spectrum; since generalized homology with respect to the sphere spectrum is homotopy, taking stable homotopy as coefficients gives a new model for the algebraic $K$-theory spectrum itself.

Recall that for an exact category $\cC$, the category $S_{q}\cC$ is a subcategory of the functor category $\Fun(\Ar[q], \cC)$ consisting of functors $F$ such that $F(i \to i) = *$ (the zero object) for all $i \in [q]$, and such that for every triple $i \leq j \leq k$ in $[q]$, the sequence 
\begin{equation*}
F(i \to j) \to F(i \to k) \to F(j \to k)
\end{equation*}
is a short exact sequence in $\cC$. Here $\Ar[q]$ is the arrow category of the ordinal $[q] = \{0 < 1 < \cdots < q\}$. Each category $S_{q}\cC$ is itself an exact category, which allows this construction to be iterated. For $q=2$, iterating the $S_2$-construction produces a cubical exact category $n\mapsto S_2^{(n)}\cC$, for each non-negative integer $n$, which McCarthy uses to define the chain complex $Q'(F,\cC)$\footnote{$Q'$ is the unnormalized version of the $Q$-construction.} associated to a functor $F$ from exact categories to chain complexes; see \cite[Section~2]{Main-McCarthy}. 

Barwick introduced \emph{exact $\infty$-categories} in \cite{Barwick_2015} as a natural $\infty$-categorical extension of Quillen’s exact categories. In this work, we extend McCarthy's $Q'$-construction to the setting of $\infty$-categories. Although McCarthy denotes this construction by $Q'$, we use the notation $Q^{\infty}$ for its $\infty$-categorical analog. We introduce a multidirectional $S_2$-construction for exact $\infty$-categories, analogous to McCarthy's iterated $S_2$-construction. Writing $T_2=\Ar[2]$, for an exact $\infty$-category $\cC$ we define $S_2^{(n)}\cC$ as a full subcategory of the functor $\infty$-category
\begin{equation*}
\Fun(\nerve(T_2^n),\cC) \ ,
\end{equation*}
obtained by imposing the exactness conditions of the $S_2$-construction in each cubical direction; see \cref{cons:S_n-construction}.
The first main property of this construction is that it preserves exactness.
\begin{theorem*}[\Cref{thm:exactness-S2C}]
Let $\cC$ be an exact $\infty$-category. Then $S_2\cC$ is also an exact $\infty$-category.
\end{theorem*}
Iterating the $S_2$-construction, for every non-negative integer $n$, we obtain a natural equivalence
\begin{equation*}
S_2^{(n)}\cC
\simeq
\underbrace{S_2(S_2(\dots(S_2\cC)\dots))}_{\text{$n$ times}}.
\end{equation*}
Consequently, $S_2^{(n)}\cC$ is an exact $\infty$-category for every $n\geq 0$; see \cref{cor:equiv-with-iteration}.

To formulate the $\infty$-categorical $Q$-construction, we first isolate the cubical combinatorics underlying McCarthy's construction. We organize the cubes $T_2^n$, together with the exact morphisms between them, into an $\infty$-category $\boxinf$. For a stable $\infty$-category $\cA$, we then consider cubical objects
\begin{equation*}
X:\boxinfop\longrightarrow \cA
\end{equation*}
and associate to each such object a cubical skeletal filtration
\begin{equation*}
\skc X
\in
\Fil_{\geq 0}
(\cA^{\boxinfop}) \ ;
\end{equation*}
see \cref{cons:skeletal-filtration}. This filtration is the central device in our construction.

The reason for passing first to a filtered object, rather than attempting to construct a coherent chain complex directly, is that the latter requires higher coherence data in arbitrarily high degrees. Indeed, a coherent chain complex in a stable $\infty$-category consists not only of objects and differentials, together with null-homotopies witnessing that consecutive differentials compose to zero, but also of higher homotopies encoding the compatibility among these choices. In Ariotta's description, this can be viewed as the additional datum of compatible trivializations of all higher Toda brackets; see \cite[Introduction and Remark~7.20]{ariotta2021coherentcochaincomplexesbeilinson}.

Filtered objects provide a convenient way to encode this coherence indirectly. By \cite[Corollary~3.4]{ariotta2021coherentcochaincomplexesbeilinson}, a filtered object may be specified by its values together with the transition morphisms between consecutive filtration degrees. Moreover, for a stable $\infty$-category $\cA$, connective filtered objects correspond to coherent connective chain complexes; see \cref{rem:filtered-to-coherent}, \cite[Theorem~4.7]{ariotta2021coherentcochaincomplexesbeilinson}, and \cite[Section~6.2]{Walde_2022}.

This relationship is compatible with Lurie's description of filtered objects by Gap objects. More precisely, Lurie's equivalence \cite[Lemma~1.2.2.4]{HigherAlgebraLurie} identifies a connective filtered object in $\cA$ with an object of
$\Gap(\ZZ_{\geq -1},\cA),$
and such a Gap object determines an ordinary chain complex in the homotopy category $\ho\cA$; see \cite[Remark~1.2.2.3]{HigherAlgebraLurie}. The compatibility between this ordinary chain complex and the coherent chain complex associated to the same filtered object is recalled in \cref{cor:coherent-chain-complex-Walde-Lurie}; see also \cite[Section~6.2]{Walde_2022}.

Applying this framework to the cubical skeletal filtration produces a coherent connective chain complex
\begin{equation*}
\skch X
:=
\cW(\skc X)
\in
\Chcoh_{\geq 0}
(\cA^{\boxinfop}) \ .
\end{equation*}
\Cref{thm:main-theorem} shows that the ordinary chain complex associated to this filtration in $\ho(\cA^{\boxinfop})$ allows us to recover the generating face maps of $X$ in $\ho\cA$. Taking their alternating sums then recovers the ordinary cubical chain complex $Q_{\bullet}(X)$.

The associated coherent chain complex also retains the original cubical object. Indeed, a coherent chain complex of cubical objects may equivalently be regarded as a cubical object of coherent chain complexes. Applying the total homology functor $\cH$ of \Cref{def:total-homology} objectwise gives \[ \widehat X:=\cH\circ\widehat{\skch X}:\boxinfop\longrightarrow\cA. \] The natural equivalence $\widehat X\simeq X$ follows formally from $\cI\cW\simeq\id$ and the pointwise stabilization of the cubical skeletal filtration. In \Cref{prop:coherent-cubical-reconstruction}, we also describe this reconstruction explicitly through total cofibers of bounded coherent chain complexes. Passing to $\ho\cA$ identifies the reconstructed cubical object with $\widetilde X$ from \Cref{thm:main-theorem}, and taking alternating sums of its face maps gives the same ordinary cubical complex; see \Cref{cor:q-prime-comparison}.

We finally return to the exact $\infty$-categorical setting. For a fixed exact $\infty$-category $\cC$, we construct a cubical object
\begin{equation*}
\square_{\cC}
:
\boxinfop
\rightarrow
\Exact_\infty
\end{equation*}
which sends $c_n$ to $S_2^{(n)}\cC$; see \cref{cons:from-box-to-exact} and \cref{cor:functor-box-to-exact}. This construction is compatible with the face functors: the opposite of each exact morphism of cubes is sent to the corresponding face functor between the iterated $S_2$-constructions; see \cref{cor:exact-morphisms-to-face-functors}.

Now let $\cA$ be a stable $\infty$-category and let
\begin{equation*}
F: \Exact_\infty \rightarrow \cA
\end{equation*}
be a functor. Composing $F$ with $\square_{\cC}$ gives a cubical object
\begin{equation*}
F\circ\square_{\cC}
:
\boxinfop
\rightarrow
\cA \ .
\end{equation*}
We define the $\infty$-categorical Mac Lane $Q$-construction to be the coherent chain complex associated to the cubical skeletal filtration of $F\circ\square_{\cC}$:
\begin{equation*}
Q^{\infty}(F;\cC)
:=
\skch
(F\circ\square_{\cC}) \ .
\end{equation*}
The construction $Q^{\infty}(F;\cC)$ is determined functorially by $F\circ\square_{\cC}$, and objectwise total homology recovers this cubical object up to natural equivalence. It therefore provides a coherent chain-complex presentation of the same cubical data. The comparison with McCarthy's ordinary construction is obtained by identifying the generating face maps through the Gap construction and then taking their alternating sums in $\ho\cA$; see \Cref{thm:main-theorem,cor:q-prime-comparison}.

Although $Q^{\infty}(F;\cC)$ is constructed homotopy-coherently, the comparison above with McCarthy's ordinary $Q'$-complex takes place in the triangulated homotopy category $\ho\cA$, rather than in an Abelian category. To obtain an Abelian setting for homological algebra, we consider the Beilinson $t$-structure associated to the cubical skeletal filtration. Its heart is identified with chain complexes in the heart of $\cA^{\boxinfop}$; see \cite{ariotta2021coherentcochaincomplexesbeilinson,antieau2024spectralsequencesdecalagebeilinson}. \Cref{rem:pointwise-t-structure} shows, however, that the pointwise $t$-structure on $\cA^{\boxinfop}$ does not recover the desired Abelian $Q'$-complex. This motivates the search for a suitable non-pointwise $t$-structure, which will be pursued in subsequent work.

Finally, McCarthy's construction has two important examples: the reduced free abelian group on the set of objects and the direct sum of the endomorphism groups over all objects, whose stabilizations recover, respectively, the integral homology of the algebraic $K$-theory spectrum and topological Hochschild homology; see \cite[Page~203]{Main-McCarthy}. In \cref{sec:examples}, we introduce two spectrum-valued analogues of these functors for exact $\infty$-categories, denoted $Z$ and $\Map$. These provide natural candidates for higher-categorical versions of McCarthy's examples, and suggest comparing the resulting $Q^{\infty}$-constructions with the corresponding $\infty$-categorical $K$-theoretic and topological Hochschild homology invariants; see \cite{Barwick_2015,BlumbergGepnerTabuada2013}. We leave this comparison to future work.


\subsection*{Organization of the material}
\label{sec:org}
\begin{itemize}
\item In \cref{sec:preliminaries}, we review exact $\infty$-categories and exact functors, coherent chain complexes in stable $\infty$-categories and their description in terms of filtered objects, and mapping spaces in additive $\infty$-categories.
\item In \cref{sec:S_n-construction}, we study the multidirectional $S_2$-construction for exact $\infty$-categories. We establish the exactness of the iterated constructions, introduce exact morphisms of cubes, and construct the face functors that determine the cubical structure used in the $Q$-construction.
\item In \cref{sec:Mac-Lane-Q-C}, we introduce the cubical skeletal filtration of a cubical object $X$ and show, using the naturality squares associated to this filtration, how the classical cubical complex $Q_\bullet(X)$ is recovered in the homotopy category. We then describe the formal reconstruction of the original cubical object from the associated coherent chain complex by objectwise total homology, including explicit total-cofiber calculations in low degrees. We next construct the cubical object in $\Exact_{\infty}$ used to define the homotopy-coherent Mac Lane construction $Q^\infty(F;\cC)$. Finally, we discuss its relation with the Beilinson $t$-structure, the obstruction arising from the pointwise $t$-structure, and spectrum-valued analogues of McCarthy's classical functors.
\end{itemize}

\section{Preliminaries}
\label{sec:preliminaries}
We recall the definition of exact $\infty$-categories and exact functors and fix a notation. We then review the $\infty$-categorical concept of chain complexes, followed by a discussion of mapping spaces in additive $\infty$-categories. Finally, we recall the terminology of the $\infty$-category of spectra and state some related miscellaneous facts. We adopt the language of $\infty$-categories as developed in \cite{HigherAlgebraLurie} and \cite{lurie2008highertopostheory}.

\subsection{Exact $\infty$-Categories}
This subsection provides an overview of the definition of exact categories and exact functors within the framework of $\infty$-categories. For future reference, we adopt Barwick's construction of the $\infty$-category of exact $\infty$-categories.

\begin{definition}
  An $\infty$-category is called pointed if it has a zero object, that is, an object that is both initial and final. For a pointed $\infty$-category $\cC$, we denote the zero object by $*$.

\end{definition}

The existence of a zero object in an $\infty$-category $\cC$ implies that for every object $x \in \cC$, the mapping spaces $\Map_{\cC}(x, *)$ and $\Map_{\cC}(*, x)$ are contractible. Consequently, for any two objects $x$ and $y$, there is a map
\begin{equation*}
    *_{xy} : x \to * \to y ,
\end{equation*}
which is unique up to a contractible space of choices. For convenience, we also denote this map by $*$.

\begin{definition}\label{def:additive-infinity-category}
    A pointed $\infty$-category $\cC$ is called \emph{pre-additive} if it admits both finite products and finite coproducts. For any $x, y \in \cC$, the natural map
    
\begin{equation*}
\begin{pmatrix}
  \Id_x & * \\
  * & \Id_y
\end{pmatrix}: x \sqcup y \longrightarrow x \times y
\end{equation*}
is an equivalence in $\cC$. We denote by $x \oplus y$ the \emph{unified product and coproduct} of $x$ and $y$. The pre-additive $\infty$-category $\cC$ is called \emph{additive} if the shear map

\begin{equation*}
s := \begin{pmatrix}
  \Id_x & \Id_x \\
  * & \Id_x
\end{pmatrix}: x \oplus x \longrightarrow x \oplus x
\end{equation*}
is an equivalence in $\cC$.
\end{definition}

\begin{remark}
    Equivalently, a pre-additive $\infty$-category $\cC$ is called additive if its homotopy category $\ho\cC$ is an ordinary additive category. This implies that $\cC$ is additive if, for every pair of objects $(x, y)$, the homotopy category $\ho\Map_{\cC}(x, y)$ is equipped with the structure of an abelian group.
\end{remark}

\begin{definition}
\label{def:exact-infinity-category} Let $\cC$ be an additive $\infty$-category. An \emph{exact structure} on $\cC$, is given by a pair $(\cM, \cE)$ of subcategories of $\cC$, where the morphisms in $\cM$ (resp.\ $\cE$) are called cofibrations (resp.\ fibrations), satisfying the following conditions (we denote $g \in \Map_\cM(x,y)$ by 
\begin{tikzcd}[cramped, row sep=small] x \arrow[r, "g", tail] & y \end{tikzcd}, and $f \in \Map_\cE(y,z)$ by 
\begin{tikzcd}[cramped, row sep=small] y \arrow[r, "f", two heads] & z \end{tikzcd}):
\begin{exaxioms}
\item\label{ax:E1} The subcategories $\cM$ and $\cE$ contain all equivalences of $\cC$. In particular, $\cM$ and $\cE$ contain all objects of $\cC$.
\item\label{ax:E2} The classes $\cM$ and $\cE$ are stable under cobase change and base change, respectively:
    \begin{exaxioms}
    \item\label{ax:E2-pushout} Morphisms in $\cM$ admit pushouts along arbitrary morphisms in $\cC$ and $\cM$ is stable under pushouts.
    \item\label{ax:E2-pullback} Morphisms in $\cE$ admit pullbacks along arbitrary morphisms in $\cC$ and $\cE$ is stable under pullbacks.
    \end{exaxioms}
\item\label{ax:E3} For any square of the form
    \begin{equation*}
    \begin{tikzcd}
        x \ar[r, "g"] \ar[d] & y \ar[d, "f"] \\
        * \ar[r] & z
    \end{tikzcd}
    \end{equation*}
    in $\cC$, we have
    \begin{exaxioms}
    \item\label{ax:E3-cocart} If $g \in \cM_1$ and the square is coCartesian, then $f \in \cE_1$ and the square is Cartesian.
    \item\label{ax:E3-cart} If $f \in \cE_1$ and the square is Cartesian, then $g \in \cM_1$ and the square is coCartesian.
    \end{exaxioms}
\end{exaxioms}
        A triple $(\cC,\cM,\cE)$ satisfying these conditions is called an \emph{exact $\infty$-category}. We often leave the choice of subcategories implicit, referring to $\cC$ as an exact $\infty$-category. 
\end{definition}

\begin{definition}

        The square
        \begin{equation*}
\begin{tikzcd}
x \arrow[r, "g"] \arrow[d, two heads] & y \arrow[d, "f"] \\
* \arrow[r, tail]                           & z                          
\end{tikzcd}
        \end{equation*}
        is called a \emph{short exact sequence} in $\cC$, if one of the following equivalent conditions holds:

        \begin{conditions}
            \item The square is coCartesian, and the map $g$ is a cofibration.
            \item The square is Cartesian, and the map $f$ is a fibration.
        \end{conditions}
        We can depict this short exact sequence as follows:
      \begin{equation*}
\begin{tikzcd}
x \arrow[r, "g", tail] & y \arrow[r, "f", two heads] & z
\end{tikzcd} \ .
\end{equation*}

   \end{definition}
\begin{definition}
   Let $\cC$ and $\cD$ be exact $\infty$-categories. A functor $F : \cC \to \cD$ is called \emph{exact} if it preserves:
   \begin{conditions}
       \item the zero object;
       \item Cofibrations and Fibrations;
       \item The pushout of cofibrations and the pullback of fibrations.
   \end{conditions}
We denote by $\ExFun(\cC, \cD)$ the full subcategory $\Fun(\cC,\cD)$ spanned by exact functors from $\cC$ to $\cD$.
\end{definition}

\begin{notation}\label{not:exact-infinity}\cite[Notation~4.2]{Barwick_2015}
   Denote by $\Exact^{\Delta}_\infty$ the following simplicial category. The objects of $\Exact^{\Delta}_\infty$ are small exact $\infty$-categories. For any two exact $\infty$-categories $\cC$ and $\cD$, let $\Exact^{\Delta}_\infty(\cC, \cD)$ denote the maximal Kan complex $\, \ExFun(\cC, \cD)^{\simeq}$ contained in $\ExFun(\cC, \cD)$. We define $\Exact_\infty$ to be the simplicial nerve \cite[Definition 1.1.5.5]{lurie2008highertopostheory} of $\Exact^{\Delta}_\infty$. In fact, $\Exact_{\infty}$ is the $\infty$-category of exact $\infty$-categories.  
\end{notation}

\subsection{Coherent Chain Complexes}\label{sec:coherent-chain-complexes}

This subsection reviews the basics of \emph{coherent chain complexes}.
For further details, we refer the reader to
\cite[Section~1.2.2]{HigherAlgebraLurie},
\cite[Sections~3, 4, 7, and~8]{ariotta2021coherentcochaincomplexesbeilinson}, and \cite[Sections~2.3 and~6.2]{Walde_2022}. Throughout this subsection, $\cA$ denotes a stable $\infty$-category.

\begin{definition}
    Let $\cC$ be an ordinary category and $S$ be a set of morphisms in $\cC$. We say $S$ is an \emph{ideal} of $\cC$, if $\cC \circ S \circ \cC \ \subseteq S $,  i.e., for every triple of composable morphisms $f$, $g$ and $h$, where $f \in S$ and $g,h \in \cC $, the composition $g \circ f \circ h$ is also a morphism in $S$.
\end{definition}

\begin{construction} \cite[Construction 2.3.1]{Walde_2022}
    We define the \emph{quotient} of $\cC$ by the ideal $S$ to be the pointed category $\frac{\cC}{S}$ as follows:
\begin{itemize}
    \item The objects of $\frac{\cC}{S}$ are the objects $x \in \cC$ along with an additional zero object $*$.
    \item The morphisms of $\frac{\cC}{S}$ are determined by setting
    \begin{equation*}
    \frac{\cC}{S}(x, y) := \frac{\cC(x, y)}{S}  \dot{\cup}  \{ x \to * \to y \}
    \end{equation*}
    for $x, y \in \cC$, with composition induced by that in $\cC$.
\end{itemize}

The category $\frac{\cC}{S}$ comes equipped with the canonical functor $\cC \to \frac{\cC}{S}$, which is the identity on objects and sends precisely the morphisms in $S$ to zero.
\end{construction}

\begin{definition}
   \cite[Definition 2.3.5]{Walde_2022} We denote by $\Ch := \frac{\ZZ}{(\rightarrow\rightarrow)}$ the quotient of the poset $\ZZ$ by the ideal $(\rightarrow\rightarrow)$ of all maps $\ \underline{n} \to \underline{m}$ with $m - n \geq 2$. Here the quotient is a pointed category: it has a zero object $*$ in addition to the objects $0,\pm1,\pm2,\ldots$. A \emph{coherent chain complex} in a pointed $\infty$-category $\cP$ is a pointed functor $\nerve(\Ch^{\op}) \to \cP$; we denote by $\Chcoh(\cP) := \Fun^*(\nerve(\Ch^{\op}), \cP)$ the $\infty$-category of coherent chain complexes in $\cP$. Similarly, we set $\Ch_{\geq 0} := \frac{\NN}{(\rightarrow\rightarrow)}$ and define the $\infty$-categories of \emph{connective chain complexes} in $\cP$ as $\Chcoh_{\geq 0}(\cP) := \Fun^*(\nerve({\Ch_{\geq 0}}^{\op}), \cP)$. 
\end{definition}

Although the definition above gives a compact description of a coherent
chain complex, constructing such a functor directly involves substantial
coherence data. Indeed, the nerve
\begin{equation*}
\nerve(\Ch_{\geq 0}^{\op})
\end{equation*}
contains simplices corresponding to arbitrarily long strings of composable
morphisms. Consequently, specifying the objects and differentials, together
with null-homotopies for the composites of consecutive differentials, is not
sufficient: one must also specify higher homotopies expressing the
compatibility of these null-homotopies. See \cite[Introduction and Section~7]{ariotta2021coherentcochaincomplexesbeilinson} for a discussion of this coherence data. For this reason, rather than constructing the coherent chain complexes appearing in this paper directly, we will construct them through \emph{filtered objects}.

\begin{definition}
Let $\cA$ be a stable $\infty$-category. An \emph{increasing filtration} in $\cA$ is a functor
\begin{equation*}
X:\nerve(\ZZ)\longrightarrow \cA,
\end{equation*}
and a \emph{decreasing filtration} in $\cA$ is a functor
\begin{equation*}
X:\nerve(\ZZ^{\op})\longrightarrow \cA.
\end{equation*}
If $\ZZ$ is replaced by $\NN$, we call the filtration
\emph{connective}. We denote by $\Fil(\cA)$ the
$\infty$-category of filtered objects in $\cA$, and by $\Fil_{\geq 0}(\cA)$ the $\infty$-category of connective
filtered objects in $\cA$. Throughout, when the direction or indexing of a filtration is clear from the context, we omit these qualifiers and simply use the term \emph{filtered object}.
\end{definition}

\begin{definition}
Let $\cA$ be a stable $\infty$-category admitting sequential limits. A decreasing filtered object
\begin{equation*}
X:\nerve(\ZZ^{\op})\longrightarrow \cA
\end{equation*}
is said to be \emph{complete} if
\begin{equation*}
\lim_{n\to+\infty} X_n \simeq 0.
\end{equation*}
\end{definition}

\begin{remark}
By reindexing along $n \mapsto -n$, a connective increasing filtered object $X$
may equivalently be regarded as a decreasing filtered object
\begin{equation*}
\widetilde{X}\colon \nerve(\ZZ^{\op})
\longrightarrow \cA
\end{equation*}
defined by
\begin{equation*}
\widetilde{X}_n \simeq
\begin{cases}
X_{-n} & n\leq 0\\
0 & n>0
\end{cases} \ .
\end{equation*}
Thus, connective filtered objects are equivalent to decreasing filtered
objects that vanish in all positive filtration degrees. In particular, every
such decreasing filtration is complete, since
\begin{equation*}
\lim_{n\to+\infty}\widetilde{X}_n \simeq 0.
\end{equation*}
Hence, we may regard every connective filtered object as a complete decreasing
filtered object.
\end{remark}

\begin{remark}\label{rem:filtered-to-coherent}
There is an equivalence
\begin{equation*}
\Fil_{\geq 0}(\cA)
\xrightarrow{\simeq}
\Chcoh_{\geq 0}(\cA);
\end{equation*}
see \cite[Section~6.2]{Walde_2022}, and \cite[Theorem~4.7]{ariotta2021coherentcochaincomplexesbeilinson}. In particular, a connective filtered
object in $\cA$ determines a coherent connective chain complex
in $\cA$. Moreover, the associated chain complex has $n$-th term
\begin{equation*}
\Omega^n\cofib(X_{n-1}\rightarrow X_n).
\end{equation*}
Thus, connective filtered objects provide a convenient way to construct
coherent connective chain complexes without specifying the higher
coherence data directly.
\end{remark}
The chain complex appearing in
\Cref{rem:filtered-to-coherent} in the homotopy category  can be described explicitly using
Lurie's Gap construction \cite[Section 1.2.2]{HigherAlgebraLurie}. 
\begin{definition}\cite[Definition 1.2.2.2]{HigherAlgebraLurie}
    Let $\cC$ be a pointed $\infty$-category, and let $I$ be a linearly ordered set.
We let $I^{[1]}$ denote the partially ordered set of pairs of elements $i \le j$ of $I$,
where $(i,j) \le (i',j')$ if $i \le i'$ and $j \le j'$.
An $I$-complex in $\cC$ is a functor
\begin{equation*}
F : \nerve(I^{[1]}) \rightarrow \cC
\end{equation*}
with the following properties:

\begin{enumerate}
    \item For each $i \in I$, the object $F(i,i)$ is a zero object of $\cC$.
    \item For every $i \le j \le k$, the associated diagram
    \begin{equation*}
\begin{tikzcd}[column sep=large, row sep=large]
    F(i,j) \arrow[r] \arrow[d] &
    F(i,k) \arrow[d] \\
    F(j,j) \arrow[r] &
    F(j,k)
    \end{tikzcd}
\end{equation*}
    is a pushout square in $\cC$.
\end{enumerate}

We let $\Gap(I,\cC)$ denote the full subcategory of
$\Fun(\nerve(I^{[1]}),\cC)$ spanned by the $I$-complexes in $\cC$.

\end{definition}

\begin{remark}\label{rem:gap-filtered-equivalence}
By \cite[Lemma~1.2.2.4]{HigherAlgebraLurie}, if
$I=I_0\cup\{-\infty\}$, where $-\infty$ is the least element of $I$,
then restriction along
\begin{equation*}
I_0\rightarrow I^{[1]},
\qquad
i\mapsto(-\infty,i),
\end{equation*}
induces an equivalence
\begin{equation*}
\Fun(\nerve(I_0),\cA)
\simeq
\Gap(I,\cA)
\end{equation*}
for any stable $\infty$-category $\cA$. Taking $I=\ZZ_{\geq -1}$ and $I_0=\ZZ_{\geq 0}=\NN$, with $-1$ as the least element of $I$, immediately yields an equivalence
\begin{equation*}
\Fil_{\geq 0}(\cA)
\simeq
\Gap(\ZZ_{\geq -1},\cA).
\end{equation*}
Moreover, if $X$ is a connective filtered object in $\cA$,
then the associated Gap object
\begin{equation*}
G_X\in
\Gap(\ZZ_{\geq -1},\cA)
\end{equation*}
admits the following explicit description. Setting $X_{-1}:=*$, the
functoriality of cofibers gives
\begin{equation*}
G_X(i,j):=
\cofib(X_i\rightarrow X_j),
\qquad
-1\leq i\leq j.
\end{equation*}
\end{remark}
\begin{remark}\label{rem:gap-associated-chain-complex}
By \cite[Remark~1.2.2.3]{HigherAlgebraLurie}, the Gap object $G_X$
constructed above determines a chain complex in the homotopy category
$\ho\cA$. Its $n$-th term is given by
\begin{equation*}
G_X(n-1,n)[-n] \simeq \Omega^n G_X(n-1,n) \simeq \Omega^n\cofib(X_{n -1}\rightarrow X_n).
\end{equation*}
Thus, this is precisely the chain complex in $\ho\cA$ appearing in \cref{rem:filtered-to-coherent}.
\end{remark}

\begin{corollary}\label{cor:coherent-chain-complex-Walde-Lurie}
With the notation above, the following diagram commutes up to a natural
equivalence:
\begin{equation*}
\begin{tikzcd}
\Fil_{\geq 0}(\cA)
\arrow[r, "\cW", "\simeq"']
\arrow[d, "\cG"']
&
\Chcoh_{\geq 0}(\cA)
\arrow[d, "U"]
\\
\Gap(\ZZ_{\geq -1},\cA)
\arrow[r, "\cL"']
&
\nerve\left(
\Ch_{\geq 0}(\ho\cA)
\right).
\end{tikzcd}
\end{equation*}
Here, $\cW$ denotes the equivalence between connective filtered
objects and coherent connective chain complexes described in
\Cref{rem:filtered-to-coherent}, while $\cG$ denotes the
associated Gap construction described in
\Cref{rem:gap-filtered-equivalence}. The functor $\cL$ is
induced by Lurie's construction of the ordinary chain complex associated
to a Gap object, as recalled in
\Cref{rem:gap-associated-chain-complex}. Finally,
\begin{equation*}
U:
\Chcoh_{\geq 0}(\cA)
\longrightarrow
\nerve\left(
\Ch_{\geq 0}(\ho\cA)
\right)
\end{equation*}
denotes the underlying chain complex functor, obtained by passing the
objects and differentials of a coherent chain complex to the homotopy
category $\ho\cA$.   
\end{corollary}
\begin{proof}
By the compatibility described in \cite[Section~6.2]{Walde_2022},
the composite $U\circ\cW$ is the ordinary chain complex associated to a connective filtered object. On the other hand, \cref{rem:gap-filtered-equivalence,rem:gap-associated-chain-complex} identify this chain complex with the one obtained by the composite $\cL\circ\cG$. Hence
\begin{equation*}
U\circ\cW
\simeq
\cL\circ\cG.
\end{equation*}
\end{proof}
\Cref{cor:coherent-chain-complex-Walde-Lurie} shows that the chain complex obtained from a Gap object in the homotopy category is not an isolated ordinary chain complex. Rather, it is the underlying ordinary
chain complex of a coherent chain complex in $\cA$. In particular, its terms and differentials are those obtained from the corresponding coherent chain complex after passing to $\ho\cA$, while the specified null-homotopies and higher
coherence data are no longer visible. This is precisely the point emphasized by Walde in \cite[Section~6.2]{Walde_2022}, where the ordinary chain complex associated to a filtered object is functorially lifted to a coherent one.

So far, we have described how a connective filtered object determines
a coherent connective chain complex.We now recall the inverse construction, following Ariotta
\cite{ariotta2021coherentcochaincomplexesbeilinson}. Using our homological indexing, denote by
\begin{equation*}
\cI \colon \Chcoh_{\geq 0}(\cA)
\xrightarrow{\simeq}
\Fil_{\geq 0}(\cA)
\end{equation*}
the inverse to the equivalence $\cW$ of
\Cref{rem:filtered-to-coherent}; see
\cite[Theorem~4.7]{ariotta2021coherentcochaincomplexesbeilinson}.

For a coherent connective chain complex $C$, let $K_n(C)$
denote the coherent $n$-cube (indexed by $(\Delta^1)^n$, as in
\cite[Proposition~3.20]{ariotta2021coherentcochaincomplexesbeilinson}) whose standard spine is
\begin{equation*}
C_n \rightarrow C_{n-1} \rightarrow \cdots
\rightarrow C_0
\end{equation*}
and whose remaining vertices are zero, with coherence data inherited
from $C$. Ariotta's description gives natural equivalences
\begin{equation*}
\cI(C)_n \simeq \totcof(K_n(C)) \ ;
\end{equation*}
see \cite[Corollary~3.22, Example~3.23, and Remark~4.9]
{ariotta2021coherentcochaincomplexesbeilinson}.
Computing this total cofiber by successive cofibers gives, for $n \geq 1$,
\begin{equation*}
\totcof(K_n(C))
\simeq
\cofib(
\Sigma^{n-1}C_n \rightarrow
\cofib(
\Sigma^{n-2}C_{n-1} \rightarrow
\cdots \rightarrow
\cofib(C_1 \rightarrow C_0)
)).
\end{equation*}
The maps appearing in this expression are induced by the
differentials and higher coherence data of $C$.
For example, the first three cases are
\begin{equation*}
\begin{aligned}
\totcof(K_0(C))
&\simeq C_0 \ , \\
\totcof(K_1(C))
&\simeq \cofib(C_1 \rightarrow C_0) \ , \\
\totcof(K_2(C))
&\simeq \cofib(
\Sigma C_2 \rightarrow
\cofib(C_1 \rightarrow C_0)) \ .
\end{aligned}
\end{equation*}
For each $n \geq 0$, the iterated cofiber construction gives
a cofiber sequence
\begin{equation*}
\Sigma^n C_{n+1}
\xrightarrow{\lambda_n}
\totcof(K_n(C))
\rightarrow
\totcof(K_{n+1}(C)) \ ,
\end{equation*}
where $\lambda_n$ is induced by the coherent structure of $C$.
Under the identifications above, the transition map
\begin{equation*}
\cI(C)_n \longrightarrow \cI(C)_{n+1}
\end{equation*}
is identified with the second map in this cofiber sequence.

\begin{corollary}\label{cor:filtered-reconstruction}
Let $F$ be a connective filtered object in $\cA$.
The filtered objects $F$ and $\cI(\mathcal W(F))$ are naturally
equivalent. Equivalences in the functor $\infty$-category
$\Fil_{\geq 0}(\cA)$ are detected pointwise,
and the preceding computation gives the following description
of their corresponding terms:
\begin{equation*}
F_n
\simeq \cI(\mathcal W(F))_n
\simeq \totcof(K_n(\mathcal W(F))) \qquad , 
\qquad n \geq 0.
\end{equation*}
\end{corollary}
\begin{definition}\label{def:total-homology}
Assume that $\cA$ admits sequential limits and sequential colimits.
Following \cite[Definition~8.8]{ariotta2021coherentcochaincomplexesbeilinson}, with our homological indexing, we define the \emph{total homology functor} to be the composite
\begin{equation*}
\cH \colon
\Chcoh_{\geq 0}(\cA) \xrightarrow{\cI}
\Fil_{\geq 0}(\cA)
\xrightarrow{\colim}
\cA \ .
\end{equation*}
Thus, for a coherent connective chain complex $C$,
\begin{equation*}
\cH(C)
:=
\colim_{n \in \NN} \cI(C)_n
\simeq
\colim_{n \in \NN}
\totcof(K_n(C)) \ ,
\end{equation*}
where the transition maps between the total cofibers are the canonical cofiber maps described above.
\end{definition}

\begin{lemma}\label{lem:bounded-total-homology}
Let $C \in \Chcoh_{\geq 0}(\cA)$, and suppose that $C_r \simeq *$ for all $r > n$, for some $n \geq 0$.
Then every transition map
\begin{equation*}
\cI(C)_r \rightarrow \cI(C)_{r+1}\ , \qquad r \geq n \ ,
\end{equation*}
is an equivalence. In particular, the restriction of
$\cI(C)$ to filtration degrees $r \geq n$ is naturally equivalent to the constant diagram with value $\cI(C)_n$. Consequently,
\begin{equation*}
\cH(C)
\simeq \cI(C)_n
\simeq \totcof(K_n(C)) \ .
\end{equation*}
These equivalences are natural in $C$ among coherent connective chain complexes vanishing in degrees greater than $n$.
\end{lemma}

\begin{proof}
For each $r \geq 0$, consider the cofiber sequence
\begin{equation*}
\Sigma^r C_{r+1} \xrightarrow{\lambda_r} \cI(C)_r \rightarrow \cI(C)_{r+1} \ .
\end{equation*}
If $r \geq n$, then $C_{r+1} \simeq *$, so the second map is an equivalence. Consequently, the structure maps
\begin{equation*}
\cI(C)_n \rightarrow \cI(C)_r \ ,
\qquad r \geq n \ ,
\end{equation*}
are equivalences. Thus, from stage $n$ onward, the filtration is constant up to equivalence, with value $\cI(C)_n$. Its colimit therefore exists and is equivalent to this value.
Hence,
\begin{equation*}
\cH(C) :=
\colim_{r \in \NN} \cI(C)_r \simeq \cI(C)_n \simeq \totcof(K_n(C)).
\end{equation*}
These identifications are natural in $C$.
\end{proof}

\subsection{Mapping Spaces in Additive $\infty$-Categories}
This subsection provides a summary of key definitions, propositions, and corollaries related to $\Einf$-groups and additive 
$\infty$-categories. For further details, we refer the reader to \cite[Chapter II, Lecture 12]{FabianLectNotes} and \cite{HigherAlgebraLurie}.

\begin{definition}
    Let $\Fin_*$ be the category of finite pointed sets, and $\cC$ be an $\infty$-category with finite products. A functor $M : \nerve(\Fin_*) \to \cC$ is called an \emph{$\Einf$-monoid} (or, simply a \emph{commutative monoid}) if the morphism 
\begin{equation*}
     M_n=M(\langle n \rangle) \overset{\simeq}{\rightarrow}  \prod_{i=1}^{n} M_1 \simeq \underbrace{M_1 \times_{*} \dots \times_{*} M_1}_{\text{$n$ times}}
\end{equation*}
    induced by $n$ inert morphisms $e_i : \langle n \rangle \to \langle 1 \rangle$ is an equivalence in $\cC$. Denote the $\infty$-category of commutative monoids by $\Mon_{\Einf}(\cC)$. We refer to the objects of $\Mon_{\Einf}(\cS)$ as $\Einf$-spaces.
(see \cite[2.0.0.2 and 2.1.1.8]{HigherAlgebraLurie} for details)
\end{definition}

\begin{remark}\label{rem:multiplication-map-of-monoids}
    Given an $\Einf$-monoid $M$ in $\cC$, the unique active morphism $f: \langle 2 \rangle\ \to \langle 1 \rangle$ induces the following coherently associative and commutative morphism
    \begin{equation*}
        m:  M_2 \simeq M_1 \times M_1  \to M_1
    \end{equation*}
    in $\cC$, which is unique up to a contractible choice of spaces. we refer to this morphism as the \emph{multiplication map}. (see \cite[2.1.2.1 and 2.4.2.3]{HigherAlgebraLurie} for more details)
\end{remark}
\begin{definition}
    Let $\cC$ be an $\infty$-category with products, and  $M \in \Mon_{\Einf}(\cC)$. We say that $M$ 
 is an \emph{$\Einf$-group} (or, simply a commutative group) if the map
    \begin{equation*}
        (\mathrm{pr}_1 , m ) : M_1 \times M_1 \to M_1 \times M_1
    \end{equation*}
    is an equivalence in $\cC$. Note that $\mathrm{pr}_1: M_1 \times M_1 \to M_1$  is the projection onto the first coordinate, and
$m$ is the multiplication map, as defined in \cref{rem:multiplication-map-of-monoids}. We denote the $\infty$-category of such commutative groups by $\Grp_{\Einf}(\cC)$. We refer to the objects of $\Grp_{\Einf}(\cS)$ as \emph{group-like $\Einf$-spaces}.

\end{definition}

\begin{proposition}
    \cite{FabianLectNotes} \label{prop:mapping-space-being-abelian-group}
    If $\cC$ is an additive $\infty$-category, then there is a canonical lift
\begin{equation*}
\Map_\cC : \cC^{\op} \times \cC \to \Grp_{\Einf}(\cS) . 
\end{equation*}
Here, the $\infty$-category of spaces $\cS$ is in fact the coherent nerve \cite[Definition 7.4.1]{Bergner_2018} $\sN(\Kan)$ of Kan complexes.
\end{proposition}

\begin{proof}
    see \cite[Proposition II.16, Corollary II.17]{FabianLectNotes} .
\end{proof}

An immediate consequence of \cref{prop:mapping-space-being-abelian-group} is that if $\cA$ is a stable $\infty$-category, then for any pair of objects $(x, y)$, the mapping space $\Map_{\cA}(x, y)$ is naturally equipped with the structure of an $\Einf$-group. The unique map $* \to \Map_{\cA}(x, y)$ determines the distinguished object $x \to * \to y$, which represents the unique zero morphism between $x$ and $y$ in $\cA$.

Moreover, for $f, g \in \Map_{\cA}(x, y)$, the multiplication map $m(f, g)$ is given by the following composition:

\begin{equation*}
  x \xrightarrow{\bigtriangleup} x \oplus x \xrightarrow{\begin{pmatrix} f & * \\ * & g \end{pmatrix}} y \oplus y \xrightarrow{\bigtriangledown} y ,  
\end{equation*}
where $\bigtriangleup$ (resp. $\bigtriangledown$) denotes the diagonal (resp. codiagonal) morphisms. We denote this composition by $f + g$.

Furthermore, the additive inverse of $f$ is given by $f \circ (-\Id_x)$, where $-\Id_x \in \Map_{\cA}(x, x)$ is the additive inverse of $\Id_x$ in the homotopy category $\ho\Map_{\cA}(x, x)$. Observe that with this definition, we have the equivalence $f + (-f) \simeq *$, where $*$ denotes the unique zero morphism in $\Map_{\cA}(x, y)$. We denote the morphism $f+(-g)$ by $f-g$. 

\begin{remark} 
  Let $(x, y)$ be a pair of objects in $\cA$. For morphisms $f, g \in \Map_{\cA}(x, y)_0$, consider the following functor:
\begin{equation*}
+ h : \Map_{\cA}(x, y) \to \Map_{\cA}(x, y), \quad f \mapsto f + h,
\end{equation*}
where $h$ is an arbitrary object in $\Map_{\cA}(x, y)$. Then, the equivalence $f \simeq g$ in $\Map_{\cA}(x, y)$ implies that $f + h$ is homotopic to $g + h$ in $\cA$. Since $f$ and $g$ are equivalent as objects of the $\infty$-category $\Map_{\cA}(x, y)$, and by the functoriality of the operation $+h$, it follows that the morphisms $f + h$ and $g + h$ are equivalent in $\Map_{\cA}(x, y)$. This implies that they are homotopic in the $\infty$-category $\cA$. In particular, by taking $h := -g$, we obtain the following homotopy in $\cA$ (or equivalently, the following equivalence in $\Map_{\cA}(x, y)$):
\begin{equation*}
f \simeq g \Rightarrow f - g := f + (-g) \simeq g + (-g) \simeq *.
\end{equation*}
This can also be seen as a conclusion of \cite[{}34.13]{Joyal}
\end{remark} 

\begin{lemma} \label{lem:additivity-composition}
   For the same maps $f$ and $g$ as above, and for any $h \in \Map_{\cA}(y, z)$, where $z \in \cA_0$ is arbitrary, the following
equivalence holds:
    \begin{equation*}
    h \circ (f+g) \simeq h \circ f + h \circ g
    \end{equation*}
\end{lemma}

\begin{proof}
   Since the $\infty$-category $\cA$ is stable, it is, by definition, also an additive $\infty$-category. By the additivity of $\cA$, we have the following in the homotopy category $\ho\cA$:
\begin{equation*}
    h \circ (f + g) = h \circ f + h \circ g.
\end{equation*}

This shows that $h \circ (f + g)$ and $h \circ f + h \circ g$ represent the same elements in the homotopy category $\ho\cA$. Therefore, they are homotopic in $\cA$, or equivalently $h \circ (f + g)$ is equivalent to $h \circ f + h \circ g$ as the objects of the $\infty$-category $\Map_{\cA}(x,z)$.

\end{proof}

\section{$\mathbf{S_n}$-Construction for Exact $\infty$-Categories} \label{sec:S_n-construction}
In this section, we extend the formalism of Mac Lane's $S$-construction for ordinary exact categories to the context of exact $\infty$-categories. We then analyze the exactness of some certain $\infty$-categories, followed by the introduction of some specific exact functors between them.
 Our approach follows the framework outlined in \cite{Main-McCarthy}.
\subsection{${S_n}$-Construction}
\begin{construction}\label{cons:S_n-construction} Let $(\cC,\mathcal M,\mathcal E)$ be an exact $\infty$-category. For $n\geq0$ and $k\geq1$, define \[ S_n^{(k)}\cC:=\operatorname{Ex}_{\infty}(\operatorname{N}(T_n^k),\cC)\subseteq\operatorname{Fun}(\operatorname{N}(T_n^k),\cC) \] to be the full subcategory spanned by the functors $F$ satisfying the following conditions: \begin{conditions} \item\label{cnd:Sn-zero} The object $F(i_1j_1,\ldots,i_kj_k)$ is a zero object whenever $i_r=j_r$ for some $1\leq r\leq k$. \item\label{cnd:Sn-cofib} For every $1\leq s\leq k$, every fixed choice of objects $t_r\in T_n$ for $r\neq s$, and all $0\leq i\leq j\leq j'\leq n$, the morphism \[ F(t_1,\ldots,t_{s-1},ij,t_{s+1},\ldots,t_k)\longrightarrow F(t_1,\ldots,t_{s-1},ij',t_{s+1},\ldots,t_k) \] is a cofibration. \item\label{cnd:Sn-fib} For every $1\leq s\leq k$, every fixed choice of objects $t_r\in T_n$ for $r\neq s$, and all $0\leq i\leq i'\leq j\leq n$, the morphism \[ F(t_1,\ldots,t_{s-1},ij,t_{s+1},\ldots,t_k)\longrightarrow F(t_1,\ldots,t_{s-1},i'j,t_{s+1},\ldots,t_k) \] is a fibration. \item\label{cnd:Sn-cocart} For every $1\leq s\leq k$, every fixed choice of objects $t_r\in T_n$ for $r\neq s$, and all $0\leq p\leq q\leq n$, the square induced by $F$, \[ \begin{tikzcd}[column sep=large,row sep=large] F(t_1,\ldots,t_{s-1},0p,t_{s+1},\ldots,t_k) \arrow[r] \arrow[d] & F(t_1,\ldots,t_{s-1},0q,t_{s+1},\ldots,t_k) \arrow[d] \\ F(t_1,\ldots,t_{s-1},pp,t_{s+1},\ldots,t_k) \arrow[r] & F(t_1,\ldots,t_{s-1},pq,t_{s+1},\ldots,t_k) \end{tikzcd} \] is cocartesian. Its lower-left vertex is a zero object by condition~\ref{cnd:Sn-zero}. \end{conditions} We set $S_n^{(0)}\cC:=\cC$ and write $S_n\cC:=S_n^{(1)}\cC$. \end{construction}

To explore this further, let us examine some examples of ${S}^{(k)}_2\cC$ for $k=1$ and 2, for a given exact $\infty$-category $\cC$. Note that $S_2^{(0)}\cC:= \cC$.
\begin{itemize}
    \item $S_2\cC$ ---
Every $F \in S_2\cC$ can be depicted in the diagram below, which represents an exact sequence in $\cC$.
\begin{equation*}
\begin{tikzcd}
F(01) \arrow[r, tail] & F(02) \arrow[r, two heads] & F(12)
\end{tikzcd}
\end{equation*}

{\item $S_2^{(2)}\cC$ --- 
Every $F \in S_2^{(2)}\cC$ can be shown as the diagram below, where all the squares are homotopy coherent squares, and the rows and columns represent exact sequences in $\cC$.
\begin{equation*}
\begin{tikzcd}
{F(01,01)} \arrow[d, tail] \arrow[r, tail]      & {F(02,01)} \arrow[d, tail] \arrow[r, two heads]      & {F(12,01)} \arrow[d, tail]      \\
{F(01,02)} \arrow[d, two heads] \arrow[r, tail] & {F(02,02)} \arrow[d, two heads] \arrow[r, two heads] & {F(12,02)} \arrow[d, two heads] \\
{F(01,12)} \arrow[r, tail]                      & {F(02,12)} \arrow[r, two heads]                      & {F(12,12)}                     
\end{tikzcd}
\end{equation*}
}
\end{itemize}

\begin{remark}

For every exact $\infty$-category $\cC$ and every $n \geq 0$, $S_2^{(n)}\cC$ is a pointed $\infty$-category. In fact, the constant functor $\underset{x \mapsto *}{*: \nerve(T^n_2)\rightarrow \cC}$ serves as the zero object.
\end{remark}

 \subsection{Exactness of $S_2^{(n)}\cC$}
\begin{remark}\label{rem:pointwise-colimit-and-limit}
    \cite[Proposition 7.1.7.3]{kerodon} Let $\cC$ be an $\infty$-category, let $S$
    be a simplicial set, and let $f: K \to \Fun(S,\cC)$ be a diagram such that,
    for every vertex $s \in S$, the composite
    $K \xrightarrow{f} \Fun(S,\cC) \xrightarrow{\operatorname{ev}_s} \cC$
    admits a colimit in $\cC$. Then:
\begin{conditions}
    \item The diagram $f$ admits a colimit in $\Fun(S, \cC)$.
    \item An extension $\bar f : K^{\triangleright} \to \Fun(S, \cC)$ of $f$
    is a colimit diagram if and only if, for each vertex $s \in S$, the composite
    $\operatorname{ev}_s \circ \bar f$ is a colimit diagram in $\cC$.
\end{conditions}
    Moreover, the dual statement holds for limit diagrams.
\end{remark}

\begin{theorem}\label{thm:exactness-S2C}
  If $\cC$ be an exact $\infty$-category, then so is $S_2\cC$. 
\end{theorem}
Before proceeding with the proof, we record the following $3\times 3$-lemma. Part~\ref{cnd:3x3-middle} is proved for exact $\infty$-categories in \cite[Lemma~1.11]{børve2023theoremretakhexactinftycategories}. The analogue of part~\ref{cnd:3x3-outer} for ordinary exact categories is proved in \cite[Corollary~3.6]{buehler2009exactcategories}. Since we need part~\ref{cnd:3x3-outer} in the $\infty$-categorical setting, we include a proof below.
    \begin{lemma}\label{lem:3x3} Let $\cC$ be an exact $\infty$-category, and consider a homotopy coherent diagram \[ \begin{tikzcd} X \arrow[r] \arrow[d, tail] & Y \arrow[r] \arrow[d, tail] & Z \arrow[d, tail] \\ X' \arrow[r, "f"] \arrow[d, two heads] & Y' \arrow[r, "g"] \arrow[d, two heads] & Z' \arrow[d, two heads] \\ X'' \arrow[r] & Y'' \arrow[r] & Z'' \end{tikzcd} \] 
whose columns are short exact sequences. Suppose that one of the following conditions holds: \begin{conditions} \item\label{cnd:3x3-middle} The middle row and one of the two outer rows are short exact. \item\label{cnd:3x3-outer} The two outer rows are short exact, and the following compatibility condition holds. Let $S_X,S_Y,S_Z\in\Fun(\Delta^1\times\Delta^1,\cC)$ be the exact squares specifying the columns. The horizontal maps extend to morphisms $a:S_X\to S_Y$ and $b:S_Y\to S_Z$, equipped with a null-homotopy \[ N:b\circ a\simeq0 \qquad\text{in }\Fun(\Delta^1\times\Delta^1,\cC), \] whose upper-left and lower-right evaluations are the specified null-homotopies exhibiting the two outer rows as short exact sequences. \end{conditions} Then the remaining row is short exact. In case~\ref{cnd:3x3-outer}, its null-homotopy is the upper-right evaluation of $N$. \end{lemma}
\begin{proof}[Proof of \ref{cnd:3x3-outer}] By \cite[Theorem~1]{KlemencStableHull}, the canonical embedding into the stable hull preserves and reflects exact sequences. It therefore suffices to work in $\cD:=\cH^{\mathrm{st}}(\cC)$, omitting the embedding from the notation. Set $K:=\cofib(f)$ and form \[ P:=\cofib(a)\in\Fun(\Delta^1\times\Delta^1,\cD). \] Since cofibers are computed pointwise, \[ P\simeq\left(\begin{tikzcd}[column sep=large, row sep=large] \cofib(X\to Y) \arrow[r] \arrow[d] & K \arrow[d] \\ * \arrow[r] & \cofib(X''\to Y'') \end{tikzcd}\right). \] Finite colimits commute, so the pointwise cofiber of the morphism between the bicartesian squares $S_X$ and $S_Y$ is again bicartesian. The morphism $b$ together with $N$ induces a morphism of squares \[ \theta:P\longrightarrow S_Z. \] Its upper-left and lower-right components are the comparison maps associated to the exact outer rows, and hence are equivalences. Using these identifications, $\theta$ gives a morphism of cofiber sequences \[ \begin{tikzcd}[column sep=large, row sep=large] Z \arrow[r] \arrow[d, "\simeq"'] & K \arrow[r] \arrow[d, "\bar g"] & Z'' \arrow[d, "\simeq"] \\ Z \arrow[r] & Z' \arrow[r] & Z''. \end{tikzcd} \] Since the two outer vertical maps are equivalences, so is $\bar g$. By construction, $\bar g$ is the comparison map induced by $g$ and the upper-right component of $N$. Thus \[ X'\xrightarrow{f}Y'\xrightarrow{g}Z' \] is a cofiber sequence in $\cD$, with the specified null-homotopy. Reflection of exact sequences along the canonical embedding now shows that the middle row is short exact in $\cC$. \end{proof}

\begin{proof}[Proof of \cref{thm:exactness-S2C}]
 By \cref{cons:S_n-construction}, we know that the objects of $S_2\cC$ are short exact sequences in $\cC$. We say that $\alpha: F \to G$ is a cofibration (resp. fibration) in $S_2\cC$ if it is a pointwise cofibration (resp. fibration) in $\cC$. Let $\cM'$ (resp. $\cE'$) denote the subcategory of $S_2\cC$ spanned by cofibrations (resp. fibrations). We now claim that, with these definitions, the $\infty$-category $(S_2\cC, \cM', \cE')$ is an exact $\infty$-category.
 
Since $\cC$ is additive and $S_2\cC$ is closed under finite biproducts computed pointwise, $S_2\cC$ is additive.

\begin{itemize}
    \item[\ref{ax:E1}]
Let $\alpha : F \to G$ be an equivalence in $S_2\cC$. By \cref{thm:pointwise-equivalence}, for every $x \in \nerve(T_2)$, the induced morphism $\alpha(x) : F(x) \to G(x)$ is an equivalence in $\cC$. This implies that each $\alpha(x)$ is a cofibration in $\cC$, so by the definition of $\cM'$, we have $\alpha \in \cM'$. Therefore, $\cM'$ contains all the equivalences in $S_2\cC$. The dual statement holds for $\cE'$.

    \item[\ref{ax:E2}]
We prove case~\ref{ax:E2-pushout}, and the proof of case~\ref{ax:E2-pullback} follows by duality. Consider the diagram $\varphi: \Lambda_0^2 \to S_2\cC$ represented below:
\begin{equation*}
\begin{tikzcd}
  F \arrow[r, tail] \arrow[d] & G \\
  H                           &  
\end{tikzcd}
\end{equation*}
Let the horizontal morphism be a cofibration in $S_2\cC$ and the vertical morphism be arbitrary. Define $\bar{\varphi} : (\Lambda_0^2)^{\triangleright} \simeq \Delta^1 \times \Delta^1 \to \Fun(\nerve(T_2),\cC)$, such that $\bar{\varphi} \mid_{\Lambda_0^2} := \varphi$ and $\bar{\varphi} \mid_{\emptyset * \Delta^0} := P$, where $P \in \Fun(\nerve(T_2), \cC)$ assigns to each object $x \in \nerve(T_2)$ the pushout of the  following diagram:
\begin{equation*}
\begin{tikzcd}
  F(x) \arrow[r, tail] \arrow[d] & G(x) \\
  H(x)                            &     
\end{tikzcd} 
\end{equation*}
By the exactness of $\cC$, such a pushout exists for each $x$. Moreover, by \cref{rem:pointwise-colimit-and-limit}, $\bar{\varphi}$ forms a colimit cone in the $\infty$-category $\Fun(\nerve(T_2), \cC)$. By taking colimits in both the vertical and horizontal directions of the homotopy coherent diagram
\begin{equation*}
\begin{tikzcd}
*                               & * \arrow[r, tail] \arrow[l]                               & *                               \\
H(01) \arrow[d, tail] \arrow[u] & F(01) \arrow[r, tail] \arrow[d, tail] \arrow[l] \arrow[u] & G(01) \arrow[d, tail] \arrow[u] \\
H(02)                           & F(02) \arrow[l] \arrow[r, tail]                           & G(02)                          
\end{tikzcd}
\end{equation*}
in both possible ways, we conclude that the square
\begin{equation*}
\begin{tikzcd}
P(01) \arrow[d] \arrow[r , "f"] & P(02) \arrow[d , "g"] \\
* \arrow[r ]               & P(12)          
\end{tikzcd}
\end{equation*}
is coCartesian. We also have the following homotopy-coherent diagram:
\begin{equation*}
\begin{tikzcd}
H(02) \arrow[d, two heads] & F(02) \arrow[l] \arrow[r, tail] \arrow[d, two heads] & G(02) \arrow[d, two heads] \\
H(12)                      & F(12) \arrow[r, tail] \arrow[l]                      & G(12)                      \\
* \arrow[u]                & * \arrow[l] \arrow[r, tail] \arrow[u]                & * \arrow[u]               
\end{tikzcd}
\end{equation*}
As a result, we obtain the equivalence $g \circ f \simeq *$. Now, consider the following homotopy coherent diagram, where $\nicefrac{G(x)}{F(x)}$ denotes the cofiber of the map $F(x) \to G(x)$ :
\begin{equation*}
\begin{tikzcd}
F(01) \arrow[r, tail] \arrow[d, tail]      & F(02) \arrow[r, two heads] \arrow[d, tail]      & F(12) \arrow[d, tail]      \\
G(01) \arrow[d, two heads] \arrow[r, tail] & G(02) \arrow[d, two heads] \arrow[r, two heads] & G(12) \arrow[d, two heads] \\
\nicefrac{G(01)}{F(01)} \arrow[r]          & \nicefrac{G(02)}{F(02)} \arrow[r]               & \nicefrac{G(12)}{F(12)}   
\end{tikzcd}
\end{equation*}
By the \cref{lem:3x3}, the bottom row of the above diagram in a short exact sequence in $\cC$. moreover, for every $x \in \nerve(T_2)$,  given the diagram 
\begin{equation*}
\begin{tikzcd}
F(x) \arrow[r] \arrow[d, tail] & H(x) \arrow[r] \arrow[d, tail] & * \arrow[d]           \\
G(x) \arrow[r]      & P(x) \arrow[r]                 & \nicefrac{G(x)}{F(x)}
\end{tikzcd}
\end{equation*}
where both the left square and the outer rectangle are pushouts. From this, it follows that the right square is also a pushout. By the exactness of $\cC$, this square is also a pullback. Consequently, we obtain short exact sequences 
\begin{equation*}
\begin{tikzcd}
H(x) \arrow[r, tail] & P(x) \arrow[r, two heads] & \nicefrac{G(x)}{F(x)}
\end{tikzcd}
\end{equation*}
 in $\cC$. Consider the following homotopy coherent diagram:
\begin{equation*}
\begin{tikzcd}
H(01) \arrow[r, tail] \arrow[d, tail]   & H(02) \arrow[r, two heads] \arrow[d, tail]   & H(12) \arrow[d, tail]      \\
P(01) \arrow[d, two heads] \arrow[r , "f"]    & P(02) \arrow[d, two heads] \arrow[r , "g"]         & P(12) \arrow[d, two heads] \\
\nicefrac{G(01)}{F(01)} \arrow[r, tail] & \nicefrac{G(02)}{F(02)} \arrow[r, two heads] & \nicefrac{G(12)}{F(12)}   
\end{tikzcd}
\end{equation*}

The exact squares displaying the sequences $H(x)\to P(x)\to\nicefrac{G(x)}{F(x)}$ assemble, by functoriality of the pointwise pushout constructions, into a functor $S:\nerve(T_2)\to\Fun(\Delta^1\times\Delta^1,\cC)$. Since $S(11)\simeq *$, its restriction to the square with vertices $01,02,11,12$ supplies the compatible null-homotopy $N$ required by part~(ii) of \Cref{lem:3x3}, whose outer evaluations are precisely the null-homotopies in the exact squares of $H$ and the pointwise cofiber $G/F$. Here, the top row, the bottom row, and all the columns are short exact sequences in $\cC$. Additionally, the equivalence $g \circ f \simeq *$ holds, and by applying \cref{lem:3x3}, we verify that the sequence
\begin{equation*}
\begin{tikzcd}
P(01) \arrow[r, "f"] & P(02) \arrow[r, "g"] & P(12)
\end{tikzcd}
\end{equation*}
is short exact in $\cC$. Therefore, the functor $P$ belongs to $S_2\cC$, which completes the proof of axiom~\ref{ax:E2} for the $\infty$-category $S_2\cC$.
    \item[\ref{ax:E3}] 
Pushouts along cofibrations and pullbacks along fibrations in $S_2\cC$ are computed pointwise by~\ref{ax:E2}. Evaluating a pushout square whose upper horizontal map is a cofibration at each $x\in T_2$, axiom~\ref{ax:E3} in $\cC$ shows that its right vertical map is a fibration and that the square is a pullback. Since $S_2\cC$ is a full subcategory of $\Fun(\nerve(T_2),\cC)$, this pointwise pullback is also a pullback in $S_2\cC$. The converse follows dually.
    \end{itemize}

\end{proof}

\begin{remark}
A morphism $\alpha: F \to G$ in $S_2\cC$ is called a fibration if it is a pointwise fibration, i.e. for each $x \in \nerve(T_2)$, the induced morphism $\alpha(x): F(x) \to G(x)$ is a fibration in $\cC$. Consequently, by the exact structure of $S_2\cC$, a short exact sequence in $S_2\cC$ can be represented by the following homotopy-coherent grid, where all rows and columns are short exact sequences in $\cC$.
\begin{equation*}
\begin{tikzcd}
X \arrow[d, tail] \arrow[r, tail]       & Y \arrow[r, two heads] \arrow[d, tail]       & Z \arrow[d, tail]       \\
X' \arrow[d, two heads] \arrow[r, tail] & Y' \arrow[d, two heads] \arrow[r, two heads] & Z' \arrow[d, two heads] \\
{X''} \arrow[r, tail]                    & {Y''} \arrow[r, two heads]                    & {Z''}                   
\end{tikzcd}
\end{equation*}

\end{remark}

\begin{lemma}\label{lem:iteration}
Let $\cC$ be an exact $\infty$-category and let $n\geq 1$. Then there is a
natural equivalence
\begin{equation*}
S_2^{(n)}\cC \simeq S_2(S_2^{(n-1)}\cC).
\end{equation*}
\end{lemma}

\begin{proof}
Using the canonical equivalence
\begin{equation*}
\nerve(T_2^n)
\simeq
\nerve(T_2)\times \nerve(T_2^{n-1}),
\end{equation*}
we have the equivalence
\begin{equation*}
\Fun(\nerve(T_2^n),\cC)
\simeq
\Fun(
\nerve(T_2),
\Fun(\nerve(T_2^{n-1}),\cC)
).
\end{equation*}
Under this equivalence, a functor $\alpha$ is sent to the functor
$\bar{\alpha}$ defined by
\begin{equation*}
\bar{\alpha}(x)(y)=\alpha(x,y),
\qquad
x\in \nerve(T_2),\quad
y\in \nerve(T_2^{n-1}).
\end{equation*}
Since the conditions defining
$\Ex_{\infty}(\nerve(T_2^n),\cC)$
are imposed coordinatewise, this equivalence restricts to an equivalence
\begin{equation*}
\Ex_{\infty}(\nerve(T_2^n),\cC)
\simeq
\Ex_{\infty}(
\nerve(T_2),
\Ex_{\infty}(\nerve(T_2^{n-1}),\cC)
).
\end{equation*}
By definition, the left-hand side is $S_2^{(n)}\cC$, while the
right-hand side is $S_2(S_2^{(n-1)}\cC)$. Hence
\begin{equation*}
S_2^{(n)}\cC \simeq S_2(S_2^{(n-1)}\cC).
\end{equation*}
\end{proof}

\begin{corollary}\label{cor:equiv-with-iteration}
 The following hold for every non-negative integer $n$:
\begin{itemize}
    \item There is the natural equivalence $S_2^{(n)}\cC \simeq \underbrace{S_2(S_2( \dots (S_2\cC) \dots ))}_{\text{$n$ times}} $.
    \item  $S_2^{(n)}\cC$ is an exact $\infty$-category, where a morphism $\alpha: F \to G$ is classified as a cofibration (resp. fibration) if, for every object $x \in \nerve({T_2}^n)$, the induced map $\alpha(x): F(x) \to G(x)$ is a cofibration (resp. fibration) in the category $\cC$.

\end{itemize}
\end{corollary}

\subsection{Exact Morphisms of Cubes}

Throughout the remainder of this paper, we use the term \emph{cube} to refer to $T_2^n$ for some integer $n\ge 1$. 

\begin{definition}
Fix a positive integer $n$, and let $k$ and $l$ be non-negative integers, where $0\leq k \leq 2$, and $1 \leq l \leq n$. For fixed $k$ and $l$ we define a functor as follows: 
\begin{equation*}
f_k(l): T_2^{n-1} \to T_2^{n}  \ ; \
\end{equation*}
\begin{equation*}
f_k(l)(i_1j_1 , i_2j_2, \dots , i_{n-1}j_{n-1})= (i_1j_1 , \dots , i_{l-1}j_{l-1} , \bar{k} , i_lj_l, \dots , i_{n-1}j_{n-1}) \qquad ; \ \ \ \ \  \ \ \bar{k}=\begin{cases}
    12 & k=0 \\
    02 & k=1 \\
    01 & k=2
\end{cases}
\end{equation*}
    We refer to every such functor as an \emph{exact morphism} of consecutive cubes. Note that there are $3n$ such functors.
    \end{definition}
   
\begin{remark}\label{rem:face-maps-to-face-functors}
Every exact morphism $f_k(l)$ defined above induces an exact functor
\begin{equation*}
d_k(l): S_2^{(n)}\cC \longrightarrow S_2^{(n-1)}\cC.
\end{equation*}
Indeed, for
$\alpha \in S_2^{(n)}\cC,$
composition with $\nerve(f_k(l))$ gives a functor
\begin{equation*}
\alpha \circ \nerve(f_k(l))
\in
\Fun(\nerve(T_2^{n-1}),\cC).
\end{equation*}
By the definition of $f_k(l)$, the functor $\alpha \circ \nerve(f_k(l))$
satisfies the conditions of \cref{cons:S_n-construction}.
Hence, it is an object of $S_2^{(n-1)}\cC$ and therefore determines the induced functor $d_k(l)$.
The exactness of $d_k(l)$ follows from
\Cref{rem:pointwise-colimit-and-limit}.
\end{remark}
\begin{definition}\label{def:face-functor-dkl}
Let $F$ be an object of $S_2^{(n)}\cC$. In view of \cref{rem:face-maps-to-face-functors}, For $1\leq l\leq n$ and
$0\leq k\leq 2$, we define the exact functor
\begin{equation*}
d_k(l):S_2^{(n)}\cC\longrightarrow S_2^{(n-1)}\cC
\end{equation*}
by
\begin{equation*}
d_k(l)(F)(i_1j_1,\dots,i_{n-1}j_{n-1})
=
F(i_1j_1,\dots,i_{l-1}j_{l-1},\bar{k},
i_lj_l,\dots,i_{n-1}j_{n-1}),
\end{equation*}
where $\bar{k}$ is as in the preceding definition. We refer to each
$d_k(l)$ as a \emph{face functor}.
\end{definition}

\begin{example}
Using \cref{def:face-functor-dkl}, there are three face functors from
$S_2\cC$ to $\cC$. These functors can be represented as follows:
\begin{equation*}
d_k(1)
\begin{tikzcd}
(F(01) \arrow[r, tail] & F(02) \arrow[r, two heads] & F(12))
\end{tikzcd}
=
\begin{cases}
F(12) & k=0,\\
F(02) & k=1,\\
F(01) & k=2.
\end{cases}
\end{equation*}

For the case $n=2$, there are six face functors from
$S_2^{(2)}\cC$ to $S_2\cC$. Let $G\in S_2^{(2)}\cC$ be represented by
\begin{equation*}
\begin{tikzcd}
G(01,01) \arrow[d, tail] \arrow[r, tail]
    & G(02,01) \arrow[d, tail] \arrow[r, two heads]
    & G(12,01) \arrow[d, tail] \\
G(01,02) \arrow[d, two heads] \arrow[r, tail]
    & G(02,02) \arrow[d, two heads] \arrow[r, two heads]
    & G(12,02) \arrow[d, two heads] \\
G(01,12) \arrow[r, tail]
    & G(02,12) \arrow[r, two heads]
    & G(12,12).
\end{tikzcd}
\end{equation*}
For example,
\begin{equation*}
d_0(1)(G)
=
\begin{tikzcd}
G(12,01) \arrow[r, tail]
    & G(12,02) \arrow[r, two heads]
    & G(12,12),
\end{tikzcd}
\end{equation*}
and
\begin{equation*}
d_1(2)(G)
=
\begin{tikzcd}
G(01,02) \arrow[r, tail]
    & G(02,02) \arrow[r, two heads]
    & G(12,02).
\end{tikzcd}
\end{equation*}
\end{example}

 \begin{proposition} \label{prop:composition-of-faces}
    For any given integers $0 \leq p,k \leq 2$ and $1\leq l < q \leq n$, the following functors are naturally equivalent in $\Fun(S_2^{(n)}\cC , S_2^{(n-2)}\cC)$:
\begin{equation*}
    d_k(l) \circ d_p(q) \simeq d_p(q-1) \circ d_k(l) .
\end{equation*}  
\end{proposition}

Before giving the proof of \cref{prop:composition-of-faces} , let us recall a crucial theorem that will be essential for our argument.

\begin{theorem} \label{thm:pointwise-equivalence} \cite[Theorem VII.1 and Theorem VII.8]{FabianLectNotes}
  Let $\cC$ and $\cD$ be $\infty$-categories. A natural transformation $\eta: F \Rightarrow G$ of functors $F, G: \cC \to \cD$ is a natural equivalence (i.e., $\eta$ is an isomorphism in the homotopy category $\ho\Fun(\cC, \cD)$) if and only if it induces equivalences on 0-simplices, i.e., $\eta_x: F(x) \to G(x)$ is an equivalence for all $x \in \cC$. 
\end{theorem}

\begin{proof}[Proof of \Cref{prop:composition-of-faces}] For $1\leq l<q\leq n$, the insertion functors satisfy \[ f_p(q)\circ f_k(l)=f_k(l)\circ f_p(q-1):T_2^{n-2}\longrightarrow T_2^n. \] Indeed, both composites insert $\bar{k}$ in position $l$ and $\bar{p}$ in position $q$, preserving the order of the remaining entries. Since $T_2^{n-2}$ and $T_2^n$ are poset categories, this agreement on objects establishes equality of the indexing functors. Precomposition with their nerves therefore induces a natural transformation \[ \eta:d_k(l)\circ d_p(q)\Longrightarrow d_p(q-1)\circ d_k(l). \] For every $F\in S_2^{(n)}\cC$, its component $\eta_F$ is the identity under the identification \[ d_k(l)d_p(q)(F)=F\circ\operatorname{N}(f_p(q)\circ f_k(l))=F\circ\operatorname{N}(f_k(l)\circ f_p(q-1))=d_p(q-1)d_k(l)(F). \] Thus every component of $\eta$ is an equivalence. By \Cref{thm:pointwise-equivalence}, $\eta$ is a natural equivalence, as required. \end{proof}

\begin{remark}\label{rem:relation-of-faces-between-cubes}
    For $n \geq 2$ and
integers $0 \leq p, k \leq 2$ and $1 \leq l < q \leq n$, we have the following
equality of functors in the ordinary category
$\Fun(T_2^{n-2}, T_2^n)$:
\begin{equation*}
f_p(q)\circ f_k(l)=f_k(l)\circ f_p(q-1) \, .
\end{equation*}
In fact, \cref{prop:composition-of-faces} follows from these relations.

\end{remark}

We now introduce the general notion of exact morphisms between any pair of cubes
$(T_2^m, T_2^n)$. More precisely, for each $m<n$, we specify what we mean by a
\emph{exact morphism} from $T_2^m$ to $T_2^n$. We then give the analogous definition
of \emph{face functors} between the exact $\infty$-categories
$S_2^{(n)}\cC$ and $S_2^{(m)}\cC$ associated to a fixed exact
$\infty$-category $\cC$. These maps will play a key role in our various definitions and constructions in the next section.

\begin{definition}[Exact morphisms of cubes]\label{def:exact-morphisms-of-cubes}
    For integers $0 \leq m<n$, an \emph{exact morphism} from $T_2^m$ to $T_2^n$ is a functor
    \begin{equation*}
f_{k_1,k_2,\dots,k_{n-m}}(l_1,l_2,\dots,l_{n-m}):T_2^m \to T_2^n
\end{equation*}
    where, for $1\leq i \leq n-m$; $k_i \in \{ 0,1,2\}$, and $1 \leq l_1<l_2<\dots <l_{n-m}\leq n$ such that for each object $(v_1,v_2,\dots,v_m) \in T_2^m$
    \begin{equation*}
f_{k_1,k_2,\dots,k_{n-m}}(l_1,l_2,\dots,l_{n-m})(v_1,v_2,\dots,v_m)=(v'_1,v'_2,\dots,v'_n) \in T_2^n
\end{equation*}
 is given as follows:
    \begin{equation*}
v'_j=\begin{cases}
     \bar{k_i} & j=l_i \\
     v_j & j<l_1 \\
     v_{j-i} & l_i<j<l_{i+1}\\
     v_{j-n+m} & j>l_{n-m}
    \end{cases}\qquad \ \ ; \ \ \ \ \ \ \ \ \  \ \ 
    \bar{k_i}=\begin{cases}
        01 & k_i=2\\
        02 & k_i=1\\
        12 & k_i=0
    \end{cases} \ .
\end{equation*}
\end{definition}
In other words, to construct an exact morphism from $T_2^m$ to $T_2^n$, we choose $n-m$ positions
\begin{equation*}
1\leq l_1<l_2<\cdots<l_{n-m}\leq n.
\end{equation*}
We place the elements $\bar{k}_1,\bar{k}_2,\dots,\bar{k}_{n-m}$ in these positions, respectively, and fill the remaining $m$ positions, from left to right, with the coordinates $v_1,\dots,v_m$. In this way, each object $(v_1,\dots,v_m)\in T_2^m$ determines an object of $T_2^n$. Since $T_2^i$'s are poset categories and the above assignment is order-preserving, it uniquely determines a functor for an arbitrary pair $m < n$.

As explained at the beginning of this section, and by the obvious analogue of
\Cref{rem:face-maps-to-face-functors}, the class of exact morphisms in
$\Fun(T_2^m, T_2^n)$ induces a corresponding class of exact functors in
$\Fun(S_2^{(n)}\cC,\, S_2^{(m)}\cC)$. These functors could be defined simply as follows:
\begin{definition}[Face functors]\label{def:face-functors}
Fix integers $0 \leq m<n$, and let
\begin{equation*}
f_{k_1,k_2,\dots,k_{n-m}}(l_1,l_2,\dots,l_{n-m})
:T_2^m\longrightarrow T_2^n
\end{equation*}
be an exact morphism of cubes as in
\Cref{def:exact-morphisms-of-cubes}.
The associated \emph{face functor} is the exact functor
\begin{equation*}
d_{k_1,k_2,\dots,k_{n-m}}(l_1,l_2,\dots,l_{n-m})
:S_2^{(n)}\cC\longrightarrow S_2^{(m)}\cC
\end{equation*}
defined by precomposition. Explicitly, for
$F\in S_2^{(n)}\cC$,
\begin{equation*}
d_{k_1,k_2,\dots,k_{n-m}}(l_1,l_2,\dots,l_{n-m})(F)
:=
F\circ
\nerve
(
f_{k_1,k_2,\dots,k_{n-m}}(l_1,l_2,\dots,l_{n-m})
),
\end{equation*}
Equivalently, for $(v_1,\dots,v_m)\in T_2^m$,
\begin{equation*}
d_{k_1,k_2,\dots,k_{n-m}}(l_1,l_2,\dots,l_{n-m})(F)
(v_1,\dots,v_m)
=
F(v'_1,\dots,v'_n),
\end{equation*}
where $(v'_1,\dots,v'_n)$ is the image of
$(v_1,\dots,v_m)$ under the exact morphism
$f_{k_1,k_2,\dots,k_{n-m}}(l_1,l_2,\dots,l_{n-m})$.
\end{definition}
In other words, a face functor is obtained by fixing $n-m$ of the
$n$ coordinates of $F$. More precisely, the coordinates in the
positions $l_1,\dots,l_{n-m}$ are fixed to
$\bar{k}_1,\dots,\bar{k}_{n-m}$, respectively, while the remaining
$m$ coordinates are allowed to vary over $T_2^m$, preserving their
order. Thus, restricting $F$ along the corresponding exact morphism
of cubes produces an object of $S_2^{(m)}\cC$.

\section{Mac Lane's $Q$-Construction for Exact $\infty$-Categories}\label{sec:Mac-Lane-Q-C}

\subsection{Skeletal Filtration of Cubical Objects}
We now extend Mac Lane's $Q$-construction to the setting of exact
$\infty$-categories. We first package the cubes and the exact morphisms
between them into an $\infty$-category $\boxinf$. Then, for any
functor
\begin{equation*}
X:\boxinfop\rightarrow\cA,
\end{equation*}
where $\cA$ is a stable $\infty$-category, we construct a
filtered object in
$\Fun(\boxinfop,\cA),$
called the \emph{cubical skeletal filtration}, and show that it determines
a chain complex in $\ho\cA$.

We then associate to every exact $\infty$-category $\cC$ a
functor
\begin{equation*}
\square_{\cC}:
\boxinfop
\rightarrow
\Exact_\infty.
\end{equation*}
Finally, we prove that, for any functor
\begin{equation*}
F:\Exact_\infty\longrightarrow\cA,
\end{equation*}
the cubical skeletal filtration of the composite
\begin{equation*}
F\circ\square_{\cC}:
\boxinfop
\longrightarrow
\cA
\end{equation*}
recovers, in the homotopy category $\ho\cA$, the chain
complex underlying Mac Lane's $Q$-construction.

\begin{definition}
    The 2-category $\square$ is defined as follows. Its objects are denoted
$c_i := T_2^i$, and for objects $c_i, c_j$ the hom-objects
$\square(c_i,c_j)$ are given by
\begin{equation*}
\square(c_i,c_j):=
\begin{cases}
    \emptyset & i>j, \\
    \Id_{c_i} & i=j, \\
    \Fun_{\Ex}(c_i,c_j) & i<j,
\end{cases}
\end{equation*}
where $\Fun_{\Ex}(c_i,c_j)\subseteq \Fun(c_i,c_j)$ denotes the
full subcategory whose objects are the exact morphisms of cubes defined in \cref{def:exact-morphisms-of-cubes}. The homotopy category of this 2-category can be depicted as follows:
\begin{equation*}
\begin{tikzcd}[ampersand replacement=\&, cramped]
	{c_0} \& {c_1} \& {c_2} \& {c_3} \& {c_4} \& {\cdots}
	\arrow[from=1-1, to=1-2]
	\arrow[shift left=0.9ex, from=1-1, to=1-2]
	\arrow[shift right=0.9ex, from=1-1, to=1-2]
	\arrow[shift left=1.8ex, from=1-2, to=1-3]
	\arrow[shift left=0.9ex, from=1-2, to=1-3]
	\arrow[shift left=0.0ex, from=1-2, to=1-3] 
	\arrow[shift right=0.9ex, from=1-2, to=1-3]
	\arrow[shift right=1.8ex, from=1-2, to=1-3]
	\arrow[shift right=2.7ex, from=1-2, to=1-3]
	\arrow[from=1-3, to=1-4]
	\arrow[shift left=0.9ex, from=1-3, to=1-4]
	\arrow[shift left=1.8ex, from=1-3, to=1-4]
	\arrow[shift left=2.7ex, from=1-3, to=1-4]
	\arrow[shift left=3.6ex, from=1-3, to=1-4]
	\arrow[shift right=0.9ex, from=1-3, to=1-4]
	\arrow[shift right=1.8ex, from=1-3, to=1-4]
	\arrow[shift right=2.7ex, from=1-3, to=1-4]
	\arrow[shift right=3.6ex, from=1-3, to=1-4]
	\arrow[shift left=1.3ex, from=1-4, to=1-5]
	\arrow[shift right=2.5ex, from=1-4, to=1-5]
	\arrow[phantom, "\vdots", from=1-4, to=1-5]
	\arrow[from=1-5, to=1-6]
\end{tikzcd}
\end{equation*}
\end{definition}

\begin{remark}\label{rem:ordinary-cubical-chain-complex}
For any additive category $\cD$, a functor $X\colon \square^{\op} \to \cD$ determines a chain complex in $\cD$ using the relations of \cref{rem:relation-of-faces-between-cubes}. Its $n$-th differential
$\delta_n \colon X_n \to X_{n-1}$ is given by the alternating sum
\begin{equation*}
\delta_n \;:=\; \sum_{i=1}^{n} (-1)^i ( X(f_0(i)) - X(f_1(i)) + X(f_2(i)) ) \ .
\end{equation*}

\end{remark}
\begin{definition}\label{def:simplicial-category-box}
  We define the simplicial category $\square^{\Delta}$ as follows:
\begin{equation*}
\Obj(\square^{\Delta}) := \Obj(\square) \;,
\end{equation*}
and for $i,j$ we set
\begin{equation*}
\Map_{\square^{\Delta}}(c_i,c_j) := \nerve(\square(c_i,c_j))^{\simeq},
\end{equation*}
where $(-)^{\simeq}$ denotes the core (i.e.\ the maximal Kan subcomplex) of a simplicial set.
 We denote the simplicial nerve \cite[Definition 1.1.5.5]{lurie2008highertopostheory} of this simplicial category by $\boxinf$.
\begin{equation*}
\sN(\square^{\Delta}):=\boxinf
\end{equation*}

\end{definition}

\begin{remark}\label{rem:discrete-mapping-spaces}
Note that, with the above definition, the mapping spaces of
$\square^\Delta$ are discrete. Since $c_j=T_2^j$ is a poset, every
isomorphism in $c_j$ is an identity. Hence, a natural isomorphism between
two exact morphisms $F,G:c_i\to c_j$ can exist only if $F=G$, and in
this case the natural isomorphism is the identity. Therefore,
$\nerve(\square(c_i,c_j))^{\simeq}$ is discrete. Consequently, by \cite[Example 1.1.5.8]{lurie2008highertopostheory}, we have
\begin{equation*}
\boxinf := \sN(\square^\Delta)
\simeq \nerve(\ho\square^\Delta)
\simeq \nerve(\ho\square).
\end{equation*}
Thus, we may regard $\boxinf$ as the nerve of the ordinary
category $\ho\square$. By a slight abuse of notation, we suppress the symbol $\ho$
whenever no confusion can arise.
\end{remark}
The reason for defining $\boxinf$ in this way is that our $Q$-construction requires a functor
\begin{equation*}
\boxinfop
\longrightarrow
\Exact_\infty.
\end{equation*}
Following \cref{not:exact-infinity},
$\Exact_\infty$ is defined as the simplicial nerve of the
simplicial category $\Exact_\infty^\Delta$. We therefore
construct the required functor by first defining a functor of simplicial
categories
\begin{equation*}
(\square^\Delta)^{\op}
\longrightarrow
\Exact_\infty^\Delta,
\end{equation*}
and then applying the simplicial nerve (see \cref{cons:from-box-to-exact}).

\begin{definition}
  For an $\infty$-category (resp. ordinary category) $\cD$, any functor 
\begin{equation*}
X : \boxinfop \to \cD
\ \text{(resp. $X: (\ho\square)^{\op} \to \cD$ )}
\end{equation*}
is called a \emph{Mac Lane cubical object} in $\cD$. We will simply refer to such a functor $X$ as a cubical object, and we write
$X_n := X(c_n)$ for each $n$. 

\end{definition}

\begin{remark}\label{rem:Q-in-homotopy-category}
  Given a cubical object $X$ in a stable $\infty$-category $\cC$, applying the homotopy category functor yields a cubical object in the homotopy category $\ho\cC$. By \cref{rem:ordinary-cubical-chain-complex}, this determines a chain complex, which we denote by $Q_{\bullet}(X)$.

\end{remark}

Based on the discussion in \cref{sec:coherent-chain-complexes}, our goal throughout the remainder of this section is to construct a filtered object in $\cA^{\square_\infty^{\mathrm{op}}}$ and prove its associated coherent chain complex allows us to recover the original cubical object $X$ and its ordinary alternating-face complex $Q_{\bullet}(X)$ in $\mathrm h\cA$.

\begin{construction}[Cubical skeletal filtration]\label{cons:skeletal-filtration}
    For $n \geq 0$, let $\square_{\infty,\leq n}$ be the full subcategory of $\boxinf$ spanned by the objects $c_0,c_1, \dots, c_n$. For a stable $\infty$-category $\cA$ and a cubical object $X: \boxinfop \to \cA$, we denote by $X_{\leq n}$ the restriction $X|_{\square_{\infty,\leq n}^{\op}}$. By \cite[Lemma 4.3.2.13]{lurie2008highertopostheory}, for every diagram
\begin{equation*}
\begin{tikzcd}[column sep=4em, row sep=4em]
\square_{\infty,\leq n}^{\op}
  \arrow[r, "X_{ \leq n}"]
  \arrow[d, hook', "i"']
&
\cA \\
\boxinfop
  \arrow[ur, dashed, "\skc_n X"']
\end{tikzcd}
\end{equation*}
there exists a left Kan extension $\skc_nX : \boxinfop \to \cA$ making the diagram commute. It is straightforward to check that $\skc_nX$ is given by
\begin{equation*}
\skc_nX(c_m) \simeq \begin{cases}
    X_m & m \leq n \\
    * & m>n 
\end{cases} \ \ \ .
\end{equation*}
Each inclusion $\square_{\infty,\leq n} \hookrightarrow \square_{\infty, \leq n+1}$ induces an inclusion $\skc_{n} X \hookrightarrow \skc_{n+1} X$ of cubical objects in the functor $\infty$-category $\cA^{\boxinfop}$. Using these functors, and applying \cite[Corollary 3.4]{ariotta2021coherentcochaincomplexesbeilinson}, we obtain a connective filtered object in the stable $\infty$-category of cubical objects in $\cA$, displayed as
\begin{equation*}
\skc X \in 
\Fil_{\geq0}(\cA^{\boxinfop}) :
\skc_0X \rightarrow 
\skc_1X \rightarrow  
\skc_2X \rightarrow \cdots .
\end{equation*}

\end{construction}

\begin{theorem}\label{thm:main-theorem} Let $\cA$ be a stable $\infty$-category and let $X:\boxinfop\to\cA$ be a cubical object. The cubical skeletal filtration $\skc X$ determines a coherent chain complex in $\cA^{\boxinfop}$. The naturality squares of its differentials determine the face maps of a cubical object $\widetilde X:\square^{\op}\to\ho\cA$, naturally isomorphic to the diagram obtained from $X$ by passing to the homotopy category. Taking alternating sums of these face maps yields a chain complex naturally isomorphic to $Q_{\bullet}(X)$. \end{theorem}
\begin{proof} Set $\skc_{-1}X:=*$. By \Cref{rem:filtered-to-coherent}, the cubical skeletal filtration $\skc X$ determines a coherent connective chain complex in $\cA^{\boxinfop}$. By \Cref{cor:coherent-chain-complex-Walde-Lurie}, its underlying ordinary chain complex can be described using the associated Gap object. More precisely, \Cref{cons:skeletal-filtration,rem:gap-filtered-equivalence} give a Gap object in $\Gap(\ZZ_{\geq-1},\cA^{\boxinfop})$, defined by \begin{equation}\label{eq:gap-object-C-ij} (i,j)\longmapsto\mathrm{C}_{i,j}X:=\cofib(\skc_iX\to\skc_jX),\qquad -1\leq i\leq j. \end{equation} In particular, $\mathrm{C}_{i,i}X\simeq*$ and $\mathrm{C}_{-1,j}X\simeq\skc_jX$. The Gap object can be depicted as follows: \[ \begin{tikzcd}[column sep=large,row sep=large] * \arrow[r] & \skc_0X \arrow[r] \arrow[d] & \skc_1X \arrow[r] \arrow[d] & \skc_2X \arrow[r] \arrow[d] & \skc_3X \arrow[r] \arrow[d] & \cdots \\ & * \arrow[r] & \mathrm{C}_{0,1}X \arrow[r] \arrow[d] \arrow[ul,phantom,"\ulcorner",very near start] & \mathrm{C}_{0,2}X \arrow[r] \arrow[d] \arrow[ul,phantom,"\ulcorner",very near start] & \mathrm{C}_{0,3}X \arrow[r] \arrow[d] \arrow[ul,phantom,"\ulcorner",very near start] & \cdots \\ & & * \arrow[r] & \mathrm{C}_{1,2}X \arrow[r] \arrow[d] \arrow[ul,phantom,"\ulcorner",very near start] & \mathrm{C}_{1,3}X \arrow[r] \arrow[d] \arrow[ul,phantom,"\ulcorner",very near start] & \cdots \\ & & & * \arrow[r] & \mathrm{C}_{2,3}X \arrow[r] \arrow[ul,phantom,"\ulcorner",very near start] & \cdots \end{tikzcd} \] By \Cref{rem:gap-associated-chain-complex}, this gives the ordinary chain complex \begin{equation}\label{eq:main-chain-homotopy-category} \skc_0X\xleftarrow{d_0}\Omega\mathrm{C}_{0,1}X\xleftarrow{d_1}\Omega^2\mathrm{C}_{1,2}X\xleftarrow{d_2}\Omega^3\mathrm{C}_{2,3}X\longleftarrow\cdots \end{equation} in $\ho(\cA^{\boxinfop})$. We use the natural transformations supplied by the Gap construction as representatives of its differentials. For each $n\geq0$, the relevant part of the Gap object extends to the following coherent diagram of bicartesian squares in $\cA^{\boxinfop}$: \[ \begin{tikzcd}[column sep=4em,row sep=3.5em] \mathrm{C}_{n-1,n}X \arrow[r] \arrow[d] & \mathrm{C}_{n-1,n+1}X \arrow[r] \arrow[d] & * \arrow[d] \\ * \arrow[r] & \mathrm{C}_{n,n+1}X \arrow[r] & \Sigma\mathrm{C}_{n-1,n}X \end{tikzcd} \] Applying the exact functor $\Omega^{n+1}$ to this entire diagram yields the cofiber sequence \[ \Omega^{n+1}\mathrm{C}_{n-1,n+1}X\longrightarrow\Omega^{n+1}\mathrm{C}_{n,n+1}X\xrightarrow{d_n}\Omega^n\mathrm{C}_{n-1,n}X, \] where the last term is identified using the suspension-loop equivalence. Since $\cA^{\boxinfop}$ is stable, this is also a fiber sequence. Consequently, it supplies a natural equivalence \[ \operatorname{fib}(d_n)\simeq\Omega^{n+1}\mathrm{C}_{n-1,n+1}X \] in the functor $\infty$-category $\cA^{\boxinfop}$. We use the comparison equivalences supplied by this fiber sequence throughout the following argument. By \Cref{cons:skeletal-filtration,rem:pointwise-colimit-and-limit}, the cubical object $\mathrm{C}_{n-1,n}X$ has value $X_n$ at $c_n$ and zero at every other cubical degree. Thus, for a generating face $f:c_n\to c_{n+1}$, the naturality square of $d_n$ identifies with \begin{equation}\label{eq:induced-squares-cubical-objects} \begin{tikzcd}[column sep=4.2em,row sep=3.2em] \Omega^{n+1}X_{n+1} \arrow[r] \arrow[d] & * \arrow[d] \\ * \arrow[r] & \Omega^nX_n \end{tikzcd} \end{equation} with its specified homotopy inherited from the natural transformation. Taking the fibers of the horizontal arrows gives the face morphism \[ \operatorname{fib}(d_n)(f):\operatorname{fib}((d_n)_{c_{n+1}})\rightarrow\operatorname{fib}((d_n)_{c_n}), \] since fibers in the functor category are computed pointwise. To identify this morphism, observe that the restriction of $\skc_{n-1}X$ to the arrow $f$ is zero, whereas the restriction of $\skc_{n+1}X$ is the arrow $X(f):X_{n+1}\to X_n$. Hence the restriction of $\mathrm{C}_{n-1,n+1}X$ to $f$ identifies naturally with $X(f)$. Evaluating the equivalence $\operatorname{fib}(d_n)\simeq\Omega^{n+1}\mathrm{C}_{n-1,n+1}X$ on $f$ therefore gives the coherent commutative diagram \begin{equation}\label{eq:simplified-diagram} \begin{tikzcd}[column sep=large,row sep=large] \operatorname{fib}((d_n)_{c_{n+1}}) \arrow[r,"\operatorname{fib}(d_n)(f)"] \arrow[d,"\simeq"'] & \operatorname{fib}((d_n)_{c_n}) \arrow[d,"\simeq"] \\ \Omega^{n+1}X_{n+1} \arrow[r,"\Omega^{n+1}X(f)"'] & \Omega^{n+1}X_n \end{tikzcd} \end{equation} The compatibility of these equivalences with the displayed morphisms follows from the fact that the comparison is an equivalence of cubical objects, not merely a collection of equivalences between their values. Applying $\Sigma^{n+1}$ and transporting along the resulting equivalences with $X_{n+1}$ and $X_n$ recovers the face map $X(f)$ in $\ho\cA$. Repeating this construction for every $n\geq0$ and every generating face recovers all generating face maps of the diagram obtained from $X$ by passing to $\ho\cA$. To express the reconstruction with suspended-loop terms, set \[ \widetilde X(c_n):=\Sigma^n\Omega^nX_n \] and define its generating face maps \[ \widetilde X(f):\Sigma^{n+1}\Omega^{n+1}X_{n+1}\longrightarrow\Sigma^n\Omega^nX_n \] by transporting the recovered morphisms along the natural suspension-loop equivalences $\Sigma^r\Omega^rX_r\simeq X_r$. Under these equivalences, the generating face maps of $\widetilde X$ identify with those of $X$. They therefore satisfy the cubical relations and define a functor $\widetilde X:\square^{\op}\to\ho\cA$. The same equivalences give a natural isomorphism between $\widetilde X$ and the diagram obtained from $X$ by passing to $\ho\cA$. Finally, by \Cref{rem:ordinary-cubical-chain-complex}, taking alternating sums of the face maps of $\widetilde X$ yields a chain complex naturally isomorphic to $Q_{\bullet}(X)$. \end{proof}

\begin{notation}
  Applying the equivalence of \cref{rem:filtered-to-coherent} to the cubical skeletal filtration, we obtain a coherent connective chain complex
\begin{equation*}
\skch X
:=
\cW(\skc X)
\in
\Chcoh_{\geq 0}
(
\cA^{\boxinfop}
) \ .
\end{equation*}
By \cref{cor:coherent-chain-complex-Walde-Lurie}, its underlying ordinary chain complex in
$\ho(\cA^{\boxinfop})$
is naturally isomorphic to the chain complex~\eqref{eq:main-chain-homotopy-category}.   
\end{notation}
\subsection{From Cubical Skeletal Filtrations to Total Homology}
The following proposition records a formal reconstruction consequence of \Cref{cor:filtered-reconstruction} and the pointwise stabilization of the cubical skeletal filtration. We include an explicit proof to describe how objectwise total homology realizes this reconstruction through total cofibers of bounded coherent chain complexes. This description shows that the required total homology is computed using finite colimits in $\cA$ and identifies the reconstructed cubical object in $\ho\cA$ with $\widetilde X$ from \Cref{thm:main-theorem}.
\begin{proposition}[Coherent cubical reconstruction]\label{prop:coherent-cubical-reconstruction}
Let $\cA$ be a stable $\infty$-category, and let
$X:\boxinfop\to\cA$ be a cubical object.
The coherent chain complex $\skch X$, regarded as a cubical diagram of coherent chain complexes, determines by objectwise total homology a cubical object
\begin{equation*}
\widehat X:\boxinfop\rightarrow\cA.
\end{equation*}
There are equivalences
\begin{equation*}
\widehat X(c_n)\simeq\Sigma^n\Omega^nX_n \, ,
    \quad n\geq0 \, ,
\end{equation*}
and a natural equivalence 
\begin{equation*}
\widehat X\simeq X
\end{equation*}
in $\cA^{\boxinfop}$. In particular, the cubical object induced by $\widehat X$ in $\ho\cA$ is naturally isomorphic to $\widetilde X$ constructed in the proof of \cref{thm:main-theorem}.
\end{proposition}

\begin{proof}
A coherent chain complex of cubical objects can also be viewed as a cubical object of coherent chain complexes:
\begin{equation*}
\begin{aligned}
\Chcoh_{\geq0}
(\cA^{\boxinfop})
&=
\Fun^*
(
\nerve(\Ch_{\geq0}^{\op}),
\cA^{\boxinfop}
)
\\
&\simeq \Fun^*
(
\nerve(\Ch_{\geq0}^{\op}),
\Fun(\boxinfop, \cA )
)
\\
&\simeq
\Fun
(
\boxinfop,
\Fun^*
(\nerve(\Ch_{\geq0}^{\op}),\cA)
)
\\
&=
\Fun
(
\boxinfop,
\Chcoh_{\geq0}(\cA)
) \ .
\end{aligned}
\end{equation*}
Under the equivalences above, we write $\widehat{\skch X}$ for $\skch X$ viewed as a cubical object of coherent chain complexes in $\cA$.

Recall that the degree-zero term of $\skch X$ is $\skc_0X$, and the degree-$i$ term, for $i\geq1$, is $\Omega^i\mathrm{C}_{i-1,i}X$. By the description of the skeletal filtration in \cref{cons:skeletal-filtration}, it follows that
\begin{equation*}
(\widehat{\skch X}(c_n))_i
\simeq
\begin{cases}
\Omega^nX_n & i=n \ ,\\
* & i\neq n \ .
\end{cases}
\end{equation*}
In particular, each $\widehat{\skch X}(c_n)$ is a bounded coherent connective chain complex.

By \cref{lem:bounded-total-homology}, the filtration associated to a bounded coherent connective chain complex is eventually constant up to equivalence. Consequently, the colimit formula of \cref{def:total-homology} defines the total homology functor on the full subcategory of such complexes in $\cA$.
Since $\cI$ and the colimits are functorial, we may precompose $\cH$ with $\widehat{\skch X}$
to obtain
\begin{equation*}
\widehat X
:=
\cH\circ
\widehat{\skch X}:
\boxinfop
\rightarrow\cA \ .
\end{equation*}
This is a coherent cubical object, since it is a composite of functors of $\infty$-categories. For $n\geq1$, \cref{lem:bounded-total-homology} and the iterated-cofiber formula give
\begin{equation*}
\begin{aligned}
\widehat X(c_n)
&:=
\cH(
\widehat{\skch X}(c_n))
\\
&\simeq
\totcof
(K_n(\widehat{\skch X}(c_n)))
\\
&\simeq
\cofib
(\Sigma^{n-1}\Omega^nX_n\rightarrow *)
\\
&\simeq
\Sigma^n\Omega^nX_n \ .
\end{aligned}
\end{equation*}
For $n=0$, the same lemma gives
$\widehat X(c_0)\simeq X_0$.

It remains to identify $\widehat X$ with $X$ as a cubical object. Applying \cref{cor:filtered-reconstruction} in the stable $\infty$-category
$\cA^{\boxinfop}$, we obtain a natural equivalence of filtered cubical objects
\begin{equation*}
\cI(\skch X)
=
\cI(\cW(\skc X))
\simeq
\skc X.
\end{equation*}
The construction of $\cI$ is computed pointwise, since its stages are given by finite cofiber constructions. Moreover, these filtrations are eventually constant at each $c_n$, so their colimits exist and are computed pointwise. We therefore have natural equivalences
\begin{equation*}
\begin{aligned}
\widehat X
&\simeq
\colim_{r\in\NN}
\cI(\skch X)_r
\\
&\simeq
\colim_{r\in\NN}
\skc_rX
\\
&\simeq X.
\end{aligned}
\end{equation*}
The last equivalence is induced by the canonical maps $\skc_rX\to X$: at $c_n$, these maps are equivalences for every $r\geq n$. Finally, \cref{thm:main-theorem} identifies $\widetilde X$ with the cubical object induced by $X$ in $\ho\cA$. The equivalence $\widehat X\simeq X$ therefore identifies the cubical object induced by $\widehat X$ in $\ho\cA$ with $\widetilde X$.
\end{proof}
\begin{example}\label{ex:coherent-cubical-reconstruction}
We illustrate this construction in cubical degrees $0$,
$1$, and $2$. Consider the first three terms of
$\skch X$:
\begin{equation*}
\skc_0X
\xleftarrow{d_0}
\Omega \mathrm{C}_{0,1}X
\xleftarrow{d_1}
\Omega^2\mathrm{C}_{1,2}X.
\end{equation*}
After evaluation at $c_2,c_1,c_0$, the objects and maps can be depicted as follows:
\begin{equation*}
\begin{tikzcd}[column sep=4em, row sep=3em]
\Omega^2X_2
  \arrow[r, shift left=2.25ex]
  \arrow[r, shift left=1.35ex]
  \arrow[r, shift left=0.45ex]
  \arrow[r, shift right=0.45ex]
  \arrow[r, shift right=1.35ex]
  \arrow[r, shift right=2.25ex]
  \arrow[d]
&
*
  \arrow[r, shift left=0.9ex]
  \arrow[r]
  \arrow[r, shift right=0.9ex]
  \arrow[d]
&
*
  \arrow[d]
\\
*
  \arrow[r, shift left=2.25ex]
  \arrow[r, shift left=1.35ex]
  \arrow[r, shift left=0.45ex]
  \arrow[r, shift right=0.45ex]
  \arrow[r, shift right=1.35ex]
  \arrow[r, shift right=2.25ex]
  \arrow[d]
&
\Omega X_1
  \arrow[r, shift left=0.9ex]
  \arrow[r]
  \arrow[r, shift right=0.9ex]
  \arrow[d]
&
*
  \arrow[d]
\\
*
  \arrow[r, shift left=2.25ex]
  \arrow[r, shift left=1.35ex]
  \arrow[r, shift left=0.45ex]
  \arrow[r, shift right=0.45ex]
  \arrow[r, shift right=1.35ex]
  \arrow[r, shift right=2.25ex]
&
*
  \arrow[r, shift left=0.9ex]
  \arrow[r]
  \arrow[r, shift right=0.9ex]
&
X_0 \ .
\end{tikzcd}
\end{equation*}
The horizontal arrows are induced by the six generating faces $c_1\to c_2$ and the three generating faces $c_0\to c_1$, respectively. The vertical arrows are the components of the differentials. The naturality homotopies and higher coherence data are inherited from
$\skch X$
and are suppressed in the display.

Each column is concentrated in degrees $0,1,2$, so its total homology is computed by the total cofiber of its associated $2$-cube. Recall that, for a coherent three-term complex $C_2\to C_1\to C_0$,
\begin{equation*}
\totcof(K_2(C))
\simeq
\cofib
(
\Sigma C_2\rightarrow
\cofib(C_1\rightarrow C_0)
) \ ,
\end{equation*}
where the map from $\Sigma C_2$ is determined by the
coherent structure of $C$. Applying this formula to the three columns gives
\begin{equation*}
\begin{aligned}
\widehat X(c_2)
&\simeq
\cofib
(
\Sigma\Omega^2X_2\rightarrow
\cofib(*\rightarrow *)
)
\\
&\simeq
\cofib
(\Sigma\Omega^2X_2\rightarrow *)
\simeq
\Sigma^2\Omega^2X_2 \ ,
\\[0.5em]
\widehat X(c_1)
&\simeq
\cofib
(
\Sigma *\rightarrow
\cofib(\Omega X_1\rightarrow *)
)
\\
&\simeq
\cofib(\Omega X_1\rightarrow *)
\simeq
\Sigma\Omega X_1 \ ,
\\[0.5em]
\widehat X(c_0)
&\simeq
\cofib
(
\Sigma *\rightarrow
\cofib(*\rightarrow X_0)
)
\\
&\simeq X_0 \ .
\end{aligned}
\end{equation*}
By functoriality of total homology, the horizontal morphisms of coherent chain complexes induce the face maps in the resulting diagram
\begin{equation*}
\begin{tikzcd}[column sep=5em]
\Sigma^2\Omega^2X_2
  \arrow[r, shift left=2.25ex]
  \arrow[r, shift left=1.35ex]
  \arrow[r, shift left=0.45ex]
  \arrow[r, shift right=0.45ex]
  \arrow[r, shift right=1.35ex]
  \arrow[r, shift right=2.25ex]
&
\Sigma\Omega X_1
  \arrow[r, shift left=0.9ex]
  \arrow[r]
  \arrow[r, shift right=0.9ex]
&
X_0 \ .
\end{tikzcd}
\end{equation*}
This diagram displays the generating face maps of
the restriction of $\widehat X$ to cubical degrees $0,1,2$. Applying $\cH$ also preserves the composition homotopies and their higher compatibilities. For example, for faces $f:c_0\to c_1$ and $g:c_1\to c_2$, it gives a homotopy
\begin{equation*}
\widehat X(f)\circ\widehat X(g)
\simeq
\widehat X(g\circ f).
\end{equation*}
Thus the construction includes the homotopies expressing the cubical identities, together with their higher coherence data. Everything takes place in $\cA$, and the natural equivalence $\widehat X\simeq X$ is compatible with this coherent structure.
\end{example}
\begin{corollary}\label{cor:q-prime-comparison}
Let $X:\boxinfop\to\cA$ be a cubical object, and let $\widetilde X$ and $\widehat X$ be as in \cref{thm:main-theorem} and \cref{prop:coherent-cubical-reconstruction}, respectively. Taking alternating sums of the face maps of
\begin{equation*}
\ho X\ ,\qquad
\widetilde X \ ,\qquad
\ho\widehat X
\end{equation*}
gives ordinary chain complexes in $\ho\cA$, all naturally isomorphic to $Q_\bullet(X)$.
\end{corollary}

\begin{proof}
The natural isomorphisms supplied by the cited results
commute with the face maps, and hence with their alternating sums.
\end{proof}

\subsection{Cubical Objects in \texorpdfstring{$\Exact_{\infty}$}{Exact-infinity}}
 The main goal of this subsection is to associate to each exact
$\infty$-category $\cC$ a cubical object in
$\Exact_\infty$. We then use this cubical object to give
our definition of the Mac Lane $Q$-construction for exact
$\infty$-categories.

\begin{construction}\label{cons:from-box-to-exact}
    Let $\alpha \in \nerve(\square(c_i,c_j))$ be an $n$-simplex. Then $\alpha$ corresponds to a functor
\begin{equation*}
\alpha \colon [n] \longrightarrow \square(c_i,c_j) \ ,
\end{equation*}
i.e. an object of $\Fun_{\Cat}([n],\square(c_i,c_j))$, and hence may be regarded as an object of
$\Fun([n],\Fun(c_i,c_j))$. Via the equivalence
\begin{equation*}
\Fun([n],\Fun(c_i,c_j)) \simeq \Fun([n]\times c_i,c_j),
\end{equation*}
we obtain a functor
\begin{equation*}
\bar{\alpha} : [n] \times c_i \to c_j \ ,
\end{equation*}
where, for each $x \in [n]$, the restriction $\bar{\alpha}|_{\{x\}\times c_i}\colon c_i \to c_j$ is an exact morphism of cubes in $\square$. Applying the nerve functor yields a map
\begin{equation*}
\nerve(\bar{\alpha}) : \Delta^n \times \nerve(c_i) \to \nerve(c_j) \ .
\end{equation*}
Finally, since for any exact $\infty$-category $\cC$ the assignment
\begin{equation*}
\Fun(-,\cC): \sSet^{\mathrm{op}} \to \Cat_{\infty}
\end{equation*}
is functorial, we obtain an induced functor between functor $\infty$-categories:
\begin{equation*}
\Fun(\nerve(\bar{\alpha}), \cC) :  \Fun(\nerve(c_j),\cC) \to \Fun(\Delta^n \times \nerve(c_i), \cC) \simeq \Fun(\Delta^n, \Fun(\nerve(c_i), \cC))
\end{equation*}
Equivalently, this determines a functor
\begin{equation*}
\overline{\Fun(\nerve(\bar{\alpha}), \cC)} : \Delta^n \times \Fun(\nerve(c_j), \cC) \to \Fun(\nerve(c_i), \cC) \ .
\end{equation*}
Restricting its domain, we obtain
\begin{equation*}
\hat{\alpha}:=\overline{\Fun(\nerve(\bar{\alpha}), \cC)}|_{\Delta^n \times \Ex_{\infty}(\nerve(c_j),\cC)} :\Delta^n \times \Ex_{\infty}(\nerve(c_j),\cC) \to  \Fun(\nerve(c_i), \cC)
\end{equation*}
Since $\alpha : c_i \to c_j$ is an exact morphism of cubes, it is straightforward to verify that $\hat{\alpha}$ carries any functor $\beta: \nerve(c_j) \to \cC$ satisfying the conditions of \cref{cons:S_n-construction} to a functor $\hat{\alpha}(\beta): \nerve(c_i) \to \cC$ satisfying the same conditions. Consequently, we obtain a functor
\begin{equation*}
\hat{\alpha}: \Delta^n \times S_2^{(j)}\cC:=\Delta^n \times \Ex_{\infty}(\nerve(c_j),\cC) \to \Ex_{\infty}(\nerve(c_i),\cC):= S_2^{(i)}\cC \ .
\end{equation*}
Moreover, by construction, for each zero simplex $x \in \Delta^n$ the induced functor $\hat{\alpha}: S_2^{(j)}\cC \to S_2^{(i)}\cC$ is an exact functor of exact $\infty$-categories. Therefore, by functoriality of the construction, we obtain a functor
\begin{equation*}
(\square^{\Delta}_{\cC})_{i,j} : \nerve(\square(c_i,c_j)) \to \ExFun(S_2^{(j)}\cC , S_2^{(i)}\cC) \ .
\end{equation*}
Applying the core functor to the latter and invoking \cref{def:simplicial-category-box} and \cref{not:exact-infinity} once again, we obtain the following functor of Kan complexes:
\begin{equation*}
(\square_{\cC}^{\Delta})_{i,j}^{\simeq}: \Map_{(\square^{\Delta})^{\op}}(c_j,c_i)=\Map_{\square^{\Delta}}(c_i,c_j)=:\nerve(\square(c_i,c_j))^{\simeq} \to \ExFun(S_2^{(j)}\cC , S_2^{(i)}\cC)^{\simeq} :=\Map_{\Exact_{\infty}^{\Delta}}(S_2^{(j)}\cC , S_2^{(i)}\cC)  \ ,
\end{equation*}
hence a functor of simplicial categories
\begin{equation}\label{eq:box-simplicial-to-exact-simplicial}
(\square_{\cC}^{\Delta})^{\simeq} : (\square^{\Delta})^{\op} \to \Exact_{\infty}^{\Delta} \ , 
\end{equation}
that  sends each $c_i$ to the exact $\infty$-category $\ S_2^{(i)}\cC$. 

\end{construction}
\begin{corollary}\label{cor:functor-box-to-exact}
    For a given exact $\infty$-category $\cC$, using that the simplicial nerve functor commutes with opposites, we obtain a functor
\begin{equation*}
\sN\bigl((\square_{\cC}^{\Delta})^{\simeq}\bigr):\boxinfop \to \Exact_{\infty}
\end{equation*}
defined by applying the simplicial nerve functor to the functor~\eqref{eq:box-simplicial-to-exact-simplicial}. We denote this functor of $\infty$-categories simply by $\square_{\cC}$.

\end{corollary}

\begin{corollary}\label{cor:exact-morphisms-to-face-functors}
\Cref{cons:from-box-to-exact} is compatible with
\Cref{def:exact-morphisms-of-cubes,def:face-functors}.
Equivalently, the functor $\square_{\cC}$ sends the opposite of
each exact morphism of cubes to the corresponding face functor of exact
$\infty$-categories. More precisely, for $m<n$,
\begin{equation*}
\square_{\cC}\left(
f_{k_1,\dots,k_{n-m}}(l_1,\dots,l_{n-m})^{\op}
\right)
\simeq
d_{k_1,\dots,k_{n-m}}(l_1,\dots,l_{n-m}).
\end{equation*}
\end{corollary}
\begin{proof}
Similar to the proof of \cref{rem:face-maps-to-face-functors}.
\end{proof}
Now, we define the $\infty$-categorical Mac Lane $Q$-construction for exact $\infty$-categories. One can find the ordinary categorical construction in the discussion right before Definition 2.1 in \cite{MacLane-Main}.

\begin{definition}
[Mac Lane $Q$-construction for exact $\infty$ categories] \label{def:Mac-Lane-Q-infinity}
Let $\cA$ be a stable $\infty$-category, and let
\begin{equation*}
F:\Exact_{\infty}\rightarrow\cA
\end{equation*}
be a functor. We define
\begin{equation*}
Q^{\infty}(F;-):
\Exact_{\infty}
\rightarrow
\Chcoh_{\geq0}
(\cA^{\boxinfop})
\end{equation*}
by setting
\begin{equation*}
Q^{\infty}(F;\cC)
:=
\skch
(F\,\square_{\cC})
\end{equation*}
for each exact $\infty$-category $\cC$.
Here $F\,\square_{\cC}$ denotes the precomposition of $F$ with 
 $\square_{\cC}$.
\end{definition}

The first three chain degrees of $Q^{\infty}(F;\cC)$ can be depicted as follows:
    \begin{equation}\label{eq:Q-infinity-first-degrees}
        \begin{tikzcd}[
  ampersand replacement=\&,
  cramped,
  row sep=2.4em,
  column sep=2.6em,
  cells={nodes={text height=1.8ex, text depth=0.6ex}}
]
	{F(\cC)} \& {*} \& {*} \& {*} \& {*} \& {\cdots} \\
	{*} \& {\Omega F(S_2\cC)} \& {*} \& {*} \& {*} \& {\cdots} \\
	{*} \& {*} \& {\Omega^2F(S_2^{(2)}\cC)} \& {*} \& {*} \& {\cdots}
	\arrow[from=1-2, to=1-1]
	\arrow[shift left=0.9ex, from=1-2, to=1-1]
	\arrow[shift right=0.9ex, from=1-2, to=1-1]
	\arrow[shift left=1.8ex, from=1-3, to=1-2]
	\arrow[shift left=0.9ex, from=1-3, to=1-2]
	\arrow[shift left=0.0ex, from=1-3, to=1-2]
	\arrow[shift right=0.9ex, from=1-3, to=1-2]
	\arrow[shift right=1.8ex, from=1-3, to=1-2]
	\arrow[shift right=2.7ex, from=1-3, to=1-2]
	\arrow[from=1-4, to=1-3]
	\arrow[shift left=0.9ex, from=1-4, to=1-3]
	\arrow[shift left=1.8ex, from=1-4, to=1-3]
	\arrow[shift left=2.7ex, from=1-4, to=1-3]
	\arrow[shift left=3.6ex, from=1-4, to=1-3]
	\arrow[shift right=0.9ex, from=1-4, to=1-3]
	\arrow[shift right=1.8ex, from=1-4, to=1-3]
	\arrow[shift right=2.7ex, from=1-4, to=1-3]
	\arrow[shift right=3.6ex, from=1-4, to=1-3]
	\arrow[shift left=2.0ex, from=1-5, to=1-4]
	\arrow[shift right=2.0ex, from=1-5, to=1-4]
	\arrow[phantom, from=1-5, to=1-4,
		"\vdots"{description, inner sep=0pt, fill=none, fill opacity=0, text opacity=1}]
	\arrow[from=1-6, to=1-5]
	\arrow[from=2-2, to=2-1]
	\arrow[shift left=0.9ex, from=2-2, to=2-1]
	\arrow[shift right=0.9ex, from=2-2, to=2-1]
	\arrow[shift left=1.8ex, from=2-3, to=2-2]
	\arrow[shift left=0.9ex, from=2-3, to=2-2]
	\arrow[shift left=0.0ex, from=2-3, to=2-2]
	\arrow[shift right=0.9ex, from=2-3, to=2-2]
	\arrow[shift right=1.8ex, from=2-3, to=2-2]
	\arrow[shift right=2.7ex, from=2-3, to=2-2]
	\arrow[from=2-4, to=2-3]
	\arrow[shift left=0.9ex, from=2-4, to=2-3]
	\arrow[shift left=1.8ex, from=2-4, to=2-3]
	\arrow[shift left=2.7ex, from=2-4, to=2-3]
	\arrow[shift left=3.6ex, from=2-4, to=2-3]
	\arrow[shift right=0.9ex, from=2-4, to=2-3]
	\arrow[shift right=1.8ex, from=2-4, to=2-3]
	\arrow[shift right=2.7ex, from=2-4, to=2-3]
	\arrow[shift right=3.6ex, from=2-4, to=2-3]
	\arrow[shift left=2.0ex, from=2-5, to=2-4]
	\arrow[shift right=2.0ex, from=2-5, to=2-4]
	\arrow[phantom, from=2-5, to=2-4,
		"\vdots"{description, inner sep=0pt, fill=none, fill opacity=0, text opacity=1}]
	\arrow[from=2-6, to=2-5]
	\arrow[from=3-2, to=3-1]
	\arrow[shift left=0.9ex, from=3-2, to=3-1]
	\arrow[shift right=0.9ex, from=3-2, to=3-1]
	\arrow[shift left=1.8ex, from=3-3, to=3-2]
	\arrow[shift left=0.9ex, from=3-3, to=3-2]
	\arrow[shift left=0.0ex, from=3-3, to=3-2]
	\arrow[shift right=0.9ex, from=3-3, to=3-2]
	\arrow[shift right=1.8ex, from=3-3, to=3-2]
	\arrow[shift right=2.7ex, from=3-3, to=3-2]
	\arrow[from=3-4, to=3-3]
	\arrow[shift left=0.9ex, from=3-4, to=3-3]
	\arrow[shift left=1.8ex, from=3-4, to=3-3]
	\arrow[shift left=2.7ex, from=3-4, to=3-3]
	\arrow[shift left=3.6ex, from=3-4, to=3-3]
	\arrow[shift right=0.9ex, from=3-4, to=3-3]
	\arrow[shift right=1.8ex, from=3-4, to=3-3]
	\arrow[shift right=2.7ex, from=3-4, to=3-3]
	\arrow[shift right=3.6ex, from=3-4, to=3-3]
	\arrow[shift left=2.0ex, from=3-5, to=3-4]
	\arrow[shift right=2.0ex, from=3-5, to=3-4]
	\arrow[phantom, from=3-5, to=3-4,
		"\vdots"{description, inner sep=0pt, fill=none, fill opacity=0, text opacity=1}]
	\arrow[from=3-6, to=3-5]
	\arrow[from=2-1, to=1-1]
	\arrow[from=2-2, to=1-2]
	\arrow[from=2-3, to=1-3]
	\arrow[from=2-4, to=1-4]
	\arrow[from=2-5, to=1-5]
	\arrow[from=2-6, to=1-6]
	\arrow[from=3-1, to=2-1]
	\arrow[from=3-2, to=2-2]
	\arrow[from=3-3, to=2-3]
	\arrow[from=3-4, to=2-4]
	\arrow[from=3-5, to=2-5]
	\arrow[from=3-6, to=2-6]
\end{tikzcd}
\end{equation}
By \cref{prop:coherent-cubical-reconstruction}, the coherent cubical object reconstructed from $Q^{\infty}(F;\cC)$ is naturally equivalent to $F\,\square_{\cC}$.
Its generating face maps can therefore be depicted as follows in $\cA$:
\begin{equation*}
\begin{tikzcd}[ampersand replacement=\&, cramped]
	{F(\cC)} \& {F(S_2\cC)} \& {F(S_2^{(2)}\cC)} \& {F(S_2^{(3)}\cC)} \& {F(S_2^{(4)}\cC)} \& {\cdots}
	\arrow[from=1-2, to=1-1]
	\arrow[shift left=0.9ex, from=1-2, to=1-1]
	\arrow[shift right=0.9ex, from=1-2, to=1-1]
	\arrow[shift left=1.8ex, from=1-3, to=1-2]
	\arrow[shift left=0.9ex, from=1-3, to=1-2]
	\arrow[shift left=0.0ex, from=1-3, to=1-2] 
	\arrow[shift right=0.9ex, from=1-3, to=1-2]
	\arrow[shift right=1.8ex, from=1-3, to=1-2]
	\arrow[shift right=2.7ex, from=1-3, to=1-2]
	\arrow[from=1-4, to=1-3]
	\arrow[shift left=0.9ex, from=1-4, to=1-3]
	\arrow[shift left=1.8ex, from=1-4, to=1-3]
	\arrow[shift left=2.7ex, from=1-4, to=1-3]
	\arrow[shift left=3.6ex, from=1-4, to=1-3]
	\arrow[shift right=0.9ex, from=1-4, to=1-3]
	\arrow[shift right=1.8ex, from=1-4, to=1-3]
	\arrow[shift right=2.7ex, from=1-4, to=1-3]
	\arrow[shift right=3.6ex, from=1-4, to=1-3]
	\arrow[shift left=2.0ex, from=1-5, to=1-4]
	\arrow[shift right=2.0ex, from=1-5, to=1-4]
	\arrow[phantom, from=1-5, to=1-4,
		"\vdots"{description, inner sep=0pt, fill=none, fill opacity=0, text opacity=1}
	]
	\arrow[from=1-6, to=1-5]
\end{tikzcd}
\end{equation*}
By \Cref{cor:q-prime-comparison}, taking alternating sums of the face maps in $\mathrm h\cA$ yields naturally isomorphic ordinary chain complexes from the original cubical object and from either reconstruction. We denote the complex associated to the original cubical object by $Q_{\bullet}(F\square_{\cC})$. Its terms are \[ Q_n(F\square_{\cC}):=F(S_2^{(n)}\cC), \qquad n\geq0, \] and its differential is \[ \delta_n:=\sum_{i=1}^{n}(-1)^i (F(d_0(i))-F(d_1(i))+F(d_2(i))), \qquad n\geq1. \] The sum is taken in the additive category $\mathrm h\cA$. This is the alternating-face formula used in McCarthy's ordinary $Q'$-construction; see \cite{Main-McCarthy}.

\begin{remark}\label{rem:passage-to-the-heart}
Assume that the stable $\infty$-category $\cA$ is equipped
with a $t$-structure. Its heart $\cA^{\heartsuit}$ is an
Abelian category; see \cite[Section~1.2.1]{HigherAlgebraLurie}. For each
integer $q$, the associated homotopy-object functor
\begin{equation*}
\pi_q:
\cA
\rightarrow
\cA^{\heartsuit}
\end{equation*}
may be applied objectwise to the cubical object
$F\square_{\cC}$, producing

\begin{equation*}
\pi_q(F\square_{\cC})
:
\boxinfop
\rightarrow
\cA^{\heartsuit}.
\end{equation*}

By \cref{rem:discrete-mapping-spaces}, the indexing $\infty$-category $\boxinf$ may be regarded as the nerve of the ordinary category $\square$. Since $\cA^{\heartsuit}$ is also equivalent to the nerve of an ordinary Abelian category, the cubical object $\pi_q(F\square_{\cC})$ is entirely ordinary categorical.
Taking alternating sums of its face morphisms therefore produces the
ordinary cubical chain complex
\begin{equation*}
Q_{\bullet}
(\pi_q(F\square_{\cC}))
\end{equation*}
in $\cA^{\heartsuit}$, recovering the corresponding
Abelian $Q'$-complex of McCarthy. Thus the passage to the heart
provides a classical shadow of the homotopy-coherent construction,
whereas $Q^{\infty}(F;\cC)$ is constructed before this passage
and retains the higher homotopical coherence present in
$\cA$.
\end{remark}

\begin{remark}\label{rem:pointwise-t-structure}
Assume that $\cA$ admits sequential limits. Using the identification of connective increasing filtered objects with complete decreasing filtered objects, and restricting the Beilinson $t$-structure to the connective subcategory, Ariotta's description gives an equivalence
\begin{equation*}
(\Fil_{\geq 0}(\cA^{\boxinfop}))^{B,\heartsuit}
\simeq
\Ch_{\geq 0}((\cA^{\boxinfop})^\heartsuit) \ ;
\end{equation*}
see \cite[Remarks~1.1 and~5.4]{ariotta2021coherentcochaincomplexesbeilinson}
and \cite[Theorem~3.21 and Definition~3.24]{antieau2024spectralsequencesdecalagebeilinson}. Moreover, under the equivalence
\begin{equation*}
\cW:
\Fil_{\geq 0}(\cA^{\boxinfop})
\xrightarrow{\simeq}
\Chcoh_{\geq 0}(\cA^{\boxinfop}) \ ,
\end{equation*}
the Beilinson $t$-structure corresponds to the levelwise $t$-structure on coherent chain complexes. In particular, Remark~5.15 of \cite{ariotta2021coherentcochaincomplexesbeilinson} gives
\begin{equation*}
\pi_q^B(\cI C)
\simeq
\pi_q^{\mathrm{lvl}}(C)
\end{equation*}
for every connective coherent chain complex $C$, after translating Ariotta's decreasing cochain conventions into our increasing chain conventions. Now equip $\cA^{\boxinfop}$ with the pointwise $t$-structure induced by a $t$-structure on $\cA$. Its heart is then 
\begin{equation*}
(\cA^{\boxinfop})^\heartsuit
\simeq
\Fun(\boxinfop,\cA^\heartsuit).
\end{equation*}
Suppose first that
\begin{equation*}
\skc X \in (\Fil_{\geq 0}(\cA^{\boxinfop}))^{B,\heartsuit} \ ,
\end{equation*}
Then
\begin{equation*}
\skch X := \cW(\skc X) \in \Ch_{\geq 0}((\cA^{\boxinfop})^\heartsuit) \ .
\end{equation*}
The degree-$n$ term of $\skch X$ is $\Omega^n\mathrm{C}_{n-1,n}X$, and
\begin{equation*}
\Omega^n\mathrm{C}_{n-1,n}X(c_i)
\simeq
\begin{cases}
\Omega^nX_n & i=n,\\
* & i\neq n.
\end{cases}
\end{equation*}
Hence
\begin{equation*}
\Omega^n\mathrm{C}_{n-1,n}X \in \Fun(\boxinfop,\cA^\heartsuit)
\end{equation*}
implies
$\Omega^nX_n\in\cA^\heartsuit,$
and therefore $X_n\in\cA^\heartsuit[n].$ Thus the terms of $Q_\bullet(X)$ lie in different shifts of the heart: the degree-$n$ term belongs to $\cA^\heartsuit[n]$, rather than to a single fixed Abelian category  $\cA^\heartsuit$. Consequently,
even before considering the differentials, $Q_\bullet(X)$ does not descend to a chain complex in one common heart under the pointwise
$t$-structure.
There is a further obstruction. Since
\begin{equation*}
X_{n+1}\in\cA^\heartsuit[n+1] \subseteq \cA_{\geq n+1} \qquad \text{and} \qquad X_n\in\cA^\heartsuit[n] \subseteq \cA_{\leq n} \ ,
\end{equation*}
the orthogonality axiom of the $t$-structure gives
\begin{equation*}
\Map_{\cA}(X_{n+1},X_n)\simeq * \ .
\end{equation*}
Consequently, the face morphisms between consecutive cubical degrees, and hence the corresponding differentials, vanish under this pointwise choice of $t$-structure. 

More generally, passing to Beilinson homotopy objects does not remove this obstruction. By Remark~5.15 of \cite{ariotta2021coherentcochaincomplexesbeilinson},
\begin{equation*}
\pi_q^B(\skc X) \simeq \pi_q^B(\cI(\cW(\skc X)))
\simeq
\pi_q^{\mathrm{lvl}}(\cW(\skc X))=\pi_q^{\mathrm{lvl}}(\skch X) \ .
\end{equation*}
Hence, considering $\pi_q^B(\skc X)$ in the Beilinson heart of $\Fil_{\geq 0} (\cA^{\boxinfop})$, the degree-$n$ term of $\pi_q^{\mathrm{lvl}}(\skch X)$ is
\begin{equation*}
\pi_q^{\mathrm{lvl}}\left(\Omega^n\mathrm{C}_{n-1,n}X\right),
\end{equation*}
and the pointwise $t$-structure gives
\begin{equation*}
\pi_q^{\mathrm{lvl}}\left(\Omega^n\mathrm{C}_{n-1,n}X\right)(c_i)
\simeq
\begin{cases}
\pi_q(\Omega^nX_n) & i=n\\
0 & i\neq n
\end{cases} \ .
\end{equation*}
Thus, in cubical degree $n$, the degree-$n$ term of
$\pi_q^{\mathrm{lvl}}(\skch X)$
is determined by
\begin{equation*}
\pi_q(\Omega^nX_n)\simeq \pi_{q+n}(X_n) \ .
\end{equation*}
As a result the pointwise $t$-structure still does not recover $Q_\bullet(\pi_qX)$: its degree-$n$ term is supported at $c_n$ with value $\pi_{q+n}(X_n)$, rather than  $\pi_q(X_n)$, and the differentials between consecutive terms vanish. 

We therefore seek a non-pointwise $t$ structure on $\cA^{\boxinfop}$ that avoids both obstructions: its homotopy-object functors should account for the cubical grading so that the degree-$n$ term recovers $\pi_q(X_n)$, and the differentials between consecutive terms should remain nontrivial. The construction of such a  $t$-structure is the subject of subsequent work.
\end{remark}

\subsection{Spectrum-valued analogues of McCarthy's functors}
\label{sec:examples}

McCarthy considers two functors on small exact categories: the reduced free abelian group on the set of objects, and the direct sum of the endomorphism groups over all objects. Their stabilizations compute, respectively, the integral homology of the algebraic $K$-theory spectrum and topological Hochschild homology; see \cite{Main-McCarthy}.

Motivated by these examples, we introduce two
spectrum-valued functors on exact $\infty$ categories. The passage from abelian groups to spectra allows these constructions to retain higher homotopical information. We use the core $\cC^\simeq$ to account for equivalences between objects.

Let $\cS_*$ denote the $\infty$-category of
pointed spaces. We write
\begin{equation*}
\Sigma^\infty:\cS_*
\rightleftarrows
\Sp:\Omega^\infty
\end{equation*}
for the suspension spectrum--infinite loop space
adjunction, with $\Sigma^\infty$ the left adjoint;
see \cite[Proposition~1.4.4.4 and its proof]
{HigherAlgebraLurie}.
\begin{definition}
For a small exact $\infty$-category $\cC$,
regard $\cC^\simeq$ as a pointed space,
pointed by a zero object. Define
\begin{equation}\label{eq:Eilenberg-Mac-Lane-functor}
Z:\Exact_{\infty}\longrightarrow\Sp \ ,
\qquad
Z(\cC):=
H\ZZ\wedge\Sigma^\infty\cC^\simeq \ ,
\end{equation}
where $H\ZZ$ denotes the Eilenberg--Mac Lane
spectrum of $\ZZ$. An exact functor $u:\cC\to\cD$
induces the map
\begin{equation*}
Z(u):=
\id_{H\ZZ}\wedge
\Sigma^\infty(u^\simeq) \ .
\end{equation*}
\end{definition}
This is a spectrum-valued analogue of the reduced
free abelian group construction, using the pointed
core in place of the pointed set of objects.

Recall that the infinite loop space functor induces
an equivalence between connective spectra and grouplike $\Einf$-spaces. We denote its inverse by
\begin{equation*}
B^\infty:
\Grp_{\Einf}(\cS)
\xrightarrow{\simeq}
\Sp_{\geq0} \ ;
\end{equation*}
see \cite[Remark~5.2.6.26]{HigherAlgebraLurie} and \cite[Theorem~5.2.6.10]{HigherAlgebraLurie}. Thus, $B^\infty M$ is the connective spectrum
associated to a grouplike $\Einf$-space
$M$.  For an exact $\infty$-category $\cC$, we apply $B^\infty$ to $\Map_{\cC}(x,x)$ equipped with its additive grouplike $\Einf$-structure.
\begin{definition}
For a small exact $\infty$-category $\cC$,
consider the diagram
\begin{equation*}
\cC^\simeq\rightarrow\Sp \ ,
\qquad
x\longmapsto
B^\infty\Map_{\cC}(x,x) \ .
\end{equation*}
An equivalence $a:x\xrightarrow{\simeq}y$ acts by
conjugation,
\begin{equation*}
f\longmapsto a\circ f\circ a^{-1},
\end{equation*}
with the coherence data induced by composition in
$\cC$. Define
\begin{equation}\label{eq:Map-functor}
\Map:
\Exact_{\infty}\rightarrow\Sp \ ,
\qquad
\Map(\cC):=
\colim_{x\in\cC^\simeq}
B^\infty\Map_{\cC}(x,x) \ ,
\end{equation}
where the colimit is taken in $\Sp$.
An exact functor induces a map of cores and compatible
maps of endomorphism spectra, and hence a map between
these colimits.
\end{definition}
This construction replaces the direct sum indexed by the set of objects with a colimit indexed by the core. These two functors provide candidates for extending McCarthy's constructions to exact $\infty$-categories.

\phantomsection  
\addcontentsline{toc}{section}{References}  
\printbibliography
\end{document}